\documentclass[pdflatex,sn-mathphys-ay]{sn-jnl}

\usepackage{graphicx}%
\usepackage{multirow}%
\usepackage{amsmath,amssymb,amsfonts}%
\usepackage{amsthm}%
\usepackage{mathrsfs}%
\usepackage[title]{appendix}%
\usepackage{xcolor}%
\usepackage{textcomp}%
\usepackage{manyfoot}%
\usepackage{booktabs}%
\usepackage{listings}%

\usepackage{enumerate}
\usepackage{natbib}
\usepackage{url} 
\usepackage[utf8]{inputenc}
\usepackage{epsfig}
\usepackage{latexsym}
\usepackage{varioref}
\usepackage{hhline}
\usepackage{mathtools}
\usepackage[inline]{enumitem}
\usepackage{longtable}
\usepackage{lineno}
\usepackage{multirow}
\usepackage{verbatim, multirow,setspace}
\usepackage{float}
\usepackage{caption}

\theoremstyle{thmstyleone}%
\theoremstyle{thmstyletwo}%

\theoremstyle{thmstylethree}%

\allowdisplaybreaks

\renewcommand{\theequation}{\arabic{equation}}
\newcommand{\beq}{\begin{equation}}
	\newcommand{\eeq}{\end{equation}}

\newtheorem{lemma}{Lemma}

\newtheorem{prop}{Proposition}

\def\ba{\begin{array}}
	\def\ea{\end{array}}

\newcommand{\rd}{\mathrm d}
\newcommand{\rE}{\mathrm E}
\newcommand{\rP}{\mathrm P}

\newcommand{\eqd}{\stackrel{d}{=}}
\newcommand{\cond}{\stackrel{d}{\rightarrow}}

\newcommand{\conp}{\stackrel{p}{\rightarrow}}

\begin{document}

\title[Confidence Intervals Using Turing's Estimator: Simulations and Applications]{Confidence Intervals Using Turing's Estimator: Simulations and Applications}

\author[1]{\fnm{Jie} \sur{Chang}}\email{changjie512@gmail.com}

\author[1]{\fnm{Michael} \sur{Grabchak}}\email{mgrabcha@charlotte.edu}

\author*[2]{\fnm{Jialin} \sur{Zhang}}\email{jzhang@math.msstate.edu}

\affil[1]{\orgdiv{Department of Mathematics and Statistics}, \orgname{UNC Charlotte}, \orgaddress{\street{9201 University City Blvd}, \city{Charlotte}, \postcode{28223}, \state{NC}, \country{USA}}}

\affil*[2]{\orgdiv{Department of Mathematics and Statistics}, \orgname{Mississippi State University}, \orgaddress{\street{75 B. S. Hood Rd}, \city{Mississippi State}, \postcode{39762}, \state{MS}, \country{USA}}}

\abstract{Consider the problem of sequentially sampling categorical data from a finite or countably infinite alphabet. After observing $n$ data points, we aim to estimate the $r$th occupancy probability, which is the total probability of all outcomes that appear exactly $r$ times in the sample. The $0$th occupancy probability, often called the missing mass, is the probability of observing a previously unseen outcome. One of the best-known estimators for these probabilities is Turing’s estimator. In this paper, we study and compare several confidence intervals (CIs) for the occupancy probabilities that are based on Turing's estimator. We perform a simulation study to better understand the finite-sample performance of these CIs and introduce an approach for selecting the appropriate CI for a given sample. We also present an application to the problem of authorship attribution and apply it to a dataset consisting of tweets from users on X (formerly Twitter). Furthermore, we derive theoretical results on the asymptotic distribution of Turing's estimator for geometric and discrete uniform distributions.}

\keywords{Turing's estimator, Good-Turing estimation, missing mass, authorship attribution, confidence intervals}

\maketitle

\section{Introduction}

Estimating the probability of unseen or rarely observed outcomes is an important problem in statistics, with applications ranging from linguistics and ecology to natural language processing and cancer research. In these contexts, data are generated sequentially from a categorical distribution over a finite or countably infinite alphabet, where the underlying distribution and sometimes even the size of the alphabet are unknown. In this context, key quantities of interest are the occupancy probabilities. Given a sample of size $n$, for any integer $0\le r\le n$, the $r$th occupancy probability is defined as the total probability of all outcomes that appear exactly $r$ times in the sample. Of particular importance is the case $r=0$, also known as the missing mass, which corresponds to the probability of seeing a previously unobserved outcome.

Turing’s estimator is one of the most widely used methods for estimating the occupancy probabilities, and its theoretical properties have been studied extensively. Many of these properties can be used to construct confidence intervals (CIs). In particular, CIs can be derived from results on the asymptotic distribution of Turing’s estimator and from concentration inequalities. However, there has not yet been a systematic effort to compare the various CIs proposed in the literature.

This paper has four main goals. First, we conduct a simulation study to evaluate and compare the finite-sample performance of the different CIs. Second, we introduce a method for choosing the appropriate CI for a given dataset. Third, we develop a novel methodology for applying these CIs to the problem of authorship attribution. We then apply this methodology to a dataset consisting of tweets from users on X (formerly Twitter). The goal is to check if two X accounts are from the same author, which is important for the problem of identifying fake accounts. Fourth, we derive theoretical results about the asymptotic distribution of Turing's estimator when sampling from discrete uniform and geometric distributions. This is important since currently available sufficient conditions are fairly complicated and verifying them in concrete situations is not trivial.

The rest of this paper is organized as follows. In Section \ref{sec:lit rev}, we give a brief overview of the literature on Turing's estimator. In Section \ref{sec: CIs}, we formally present Turing's estimator, review the various asymptotic results available in the literature, and convert them into CIs. We also introduce the so-called Heuristic CI, which alternates between the others depending on the properties of a given random sample. In Section \ref{sec: sims}, we give our main simulation results with a focus on three classes of distributions: discrete uniform, geometric, and discrete Pareto. To give the simulations more context, we also present our theoretical results here. In Section \ref{sec: concentration ineq CIs}, we extend our simulation study to compare the performance of the asymptotic CIs with concentration inequality-based ones. In Section \ref{sec: application}, we develop a novel methodology for authorship attribution and apply it to a dataset consisting of tweets from users on X. A summary along with a discussion is given in Section \ref{sec: disc}. Appendix \ref{sec: cond} summarizes the known sufficient conditions for the asymptotic distribution of Turing's estimator and Appendix \ref{sec: proofs} presents the proofs of our theoretical results.

\section{Literature Review}\label{sec:lit rev}

The problem of observing categorical data on a finite or countably infinite alphabet arises in many applications. The letters in this alphabet can represent different things depending on the context. In this paper, we focus on applications to linguistics, where letters represent words used by an author \citep{Zhang:Huang:2007}. In other applications, they may represent species in an ecosystem \citep{Chao:etal:2015}, anatomical locations of metastatic tumors \citep{Newton:etAl:2014}, expressed sequence tags (ESTs) for genetic identification \citep{Lijoi:Mena:Prunster:2007}, word trigrams for speech recognition \citep{Gupta:Lennig:Mermelstein:1992}, colors of balls in an urn \citep{Decrouez:Grabchak:Paris:2016}, or different fortunes placed into fortune cookies \citep{Gou:Ruth:Basickes:Litwin:2024}. In all of these contexts, Turing's estimator can be used to estimate the occupancy probabilities, and, what is often more interesting, the missing mass.

Turing’s estimator, also known as Turing’s formula or the Good-Turing estimator, was first published in \cite{Good:1953}, with the idea largely credited to Alan Turing. Turing had developed it while working to break the Enigma code during World War II; see \cite{Good:2000} for a discussion of its applications in this context. Although the formula is simple to state, it is not immediately clear how or why it works. According to \cite{Good:2000}, Turing had provided a straightforward demonstration to justify it, but Good later forgot the details, and the explanation has been lost to history. Instead, \cite{Good:1953} offers a justification based on what is now known as empirical Bayes. Since then, there has been much work on understanding the statistical properties of Turing's estimator.

One of the earliest papers to study this is \cite{Robbins:1968}, which considers the bias of Turing’s estimator and shows that it would be unbiased if we had an additional observation. Other studies of the bias include \cite{Chao:Lee:Chen:1988} and \cite{Zhang:Huang:2007}, which independently develop a modification of Turing's estimator with reduced bias. Studies of the mean square error (MSE) of Turing's estimator can be found in \cite{Painsky:2022} and the references therein. Consistency of Turing's estimator is studied in \cite{Ohannessian:Dahleh:2012}, where it is emphasized that the standard notion of consistency is not appropriate because the quantity being estimated is not fixed but converging to zero as the sample size increases. Instead, they propose analyzing consistency in terms of relative error. A simulation study aimed at understanding the finite-sample performance of Turing's estimator can be found in \cite{Grabchak:Cosme:2017}. 

Our interest is in CIs for the occupancy probabilities. There are two main approaches: the first is based on asymptotic distributions and the second on concentration inequalities. Perhaps the first papers to study the asymptotic distribution of Turing's estimator are \cite{Esty:1982} and \cite{Esty:1983}, which give sufficient conditions for asymptotic Poissonity and asymptotic normality, respectively, in the case $r=0$. More general sufficient conditions, in this case, are given in \cite{Zhang:Huang:2008} and then necessary and sufficient conditions are given in \cite{Zhang:Zhang:2009}. Extensions of these results to any $r\ge0$ are given in \cite{Zhang:2013}, \cite{Grabchak:Zhang:2017}, and \cite{Chang:Grabchak:2023}. The study of concentration inequalities for Turing's formula was initiated in \cite{McAllester:Schapire:2000}. More recent results can be found in \cite{Painsky:2023}, \cite{Favaro:Naulet:2024}, and the references therein. 

A monographic treatment of the statistical properties of Turing's estimator can be found in \cite{Zhang:2017}. More recently, variants of Turing's estimator have been introduced to estimate certain quantities beyond the missing mass and the occupancy probabilities, see, e.g., \cite{Painsky2023large}, \cite{Favaro:Naulet:2024}, and \cite{Chandra:2024}. In a different direction, Turing's estimator has been applied in a context where the data is not independent, but follows a Markovian structure, see, e.g., \cite{Pananjady:Muthukumar:Thangaraj:2024} and the references therein.

\section{Confidence Intervals for Occupancy Probabilities}\label{sec: CIs}

Let $\mathcal{A}$ be a finite or countably infinite set. We refer to $\mathcal A$ as the alphabet and to the elements in $\mathcal{A}$ as letters. Let $\mathcal{P}= \{ p_{\ell}: \ell \in \mathcal{A}\}$, with $\sum_{\ell\in\mathcal A}p_\ell=1$ and $p_\ell\in[0,1]$ for each $\ell\in\mathcal A$, be a probability distribution on $\mathcal A$.  In practice, we do not know the distribution $\mathcal P$, which must be estimated from data. Let  $X_1,X_2,\dots,X_n$ be a random sample of size $n$ on alphabet $\mathcal{A}$ with distribution $\mathcal{P}$.  For each $\ell \in \mathcal{A}$, let $y_{\ell,n} =\sum_{i=1}^{n}1_{\left[X_i = \ell\right]}$ be the number of times that letter $\ell$ appears in the sample. For $r = 0,1,2,\dots,n$, define
\begin{align*}
	\pi_{r,n}=\sum_{\ell \in \mathcal{A}}p_{\ell} 1_{[y_{\ell,n}=r]}    \hspace{0.2cm}\mbox{and}\hspace{0.2cm}
	N_{r,n}= \sum_{\ell \in \mathcal{A}} 1_{[y_{\ell,n} =r]}.
\end{align*}
Here, $\pi_{r,n}$ is the conditional probability (given the sample) of seeing a letter that appears exactly $r$ times in the sample and $ N_{r,n}$ is the number of letters that appear exactly $r$ times. Note that $0\le \pi_{r,n}\le1$ and that, for $r\ge1$, we have $0\le N_{r,n}\le n/r$. We refer to $\pi_{r,n}$ as the $r$th occupancy probability and to $N_{r,n}$ as the $r$th occupancy count. When $r=0$, the quantity $\pi_{0,n}$ is also called the missing mass.  It represents the probability of seeing a letter that has not been seen previously. The quantity $1-\pi_{0,n}$ is called the coverage of the random sample. When this quantity is small, the sample is not representative and does not provide much information about the underlying distribution; see \cite{Zhang:Grabchak:2013} for a discussion.

For $r= 0, 1, 2,\dots, (n-1)$ the $r$th-order Turing estimator is an estimator of $\pi_{r,n}$ given by
\begin{align*}
	T_{r,n}=\frac{N_{r+1,n}}{n}(r+1).
\end{align*}
Its bias is given by
\begin{eqnarray}\label{eq: bias}
	\rE\left[T_{r,n}- \pi_{r,n}\right] = \binom{n}{r} \sum_{\ell\in\mathcal A}p_{\ell}^{r+1}(1-p_{\ell})^{n-r-1}\left(p_{\ell}-\frac{r}{n}\right).
\end{eqnarray}
The asymptotic standard deviation of $n\left(T_{r,n}-\pi_{r,n}\right)$ can be written as
\begin{eqnarray}\label{eq: asymptotic sd}
	s_{r,n}=\sqrt{ \sum_{\ell\in\mathcal A} (r+1+np_{\ell}) e^{-np_{\ell}} \frac{(np_{\ell})^{r+1}}{r!}},
\end{eqnarray}
and it can be estimated by
\begin{align*}
	\hat{s}_{r,n} =\sqrt{(r+1)^2 N_{r+1,n} + (r+2)(r+1) N_{r+2,n}},
\end{align*}
see \cite{Chang:Grabchak:2023} for details.

We are interested in studying confidence intervals (CIs) for $\pi_{r,n}$. This is a somewhat unusual quantity to construct CIs for, as it is neither a parameter (since it depends on the random sample) nor a statistic (since it depends on unknown parameters). Nevertheless, one can interpret a CI for this quantity in the usual frequentist manner. Specifically, if we repeatedly obtain random samples of size $n$ and for each sample we construct a $(1-\alpha)100\%$ CI, then, in the long run, $\pi_{r,n}$ should be in $(1-\alpha)100\%$ of these CIs. The difference is that $\pi_{r,n}$ depends on the sample and is thus a different quantity for each random sample. While no exact CIs for $\pi_{r,n}$ are known, several approximate CIs have been proposed in the literature. These are based either on asymptotic results or on concentration inequalities. We are primarily focused on asymptotic CIs, but see Section \ref{sec: concentration ineq CIs} for a discussion of concentration inequality-based CIs. Asymptotic CIs are based on various limit theorems, which we now recall.

If $s_{r,n}\to\infty$, then, under mild conditions (see Appendix \ref{sec: cond}), we have
\begin{eqnarray}\label{eq: asymp normal}
	\frac{n}{\hat s_{r,n}}\left(T_{r,n}-\pi_{r,n}\right) \cond N(0,1)
\end{eqnarray}
and 
\begin{eqnarray}\label{eq: asymp ratio normal}
	T_{r,n} \frac{n}{\hat s_{r,n}}\left(\frac{T_{r,n}}{\pi_{r,n}}-1 \right)\cond N(0,1).
\end{eqnarray}
These lead to the $(1-\alpha)100\%$ CIs for $\pi_{r,n}$ given, respectively, by
\begin{eqnarray}\label{eq: Normal CI}
	\left[T_{r,n}-z_{\alpha/2}\frac{\hat{s}_{r,n}}{n}, T_{r,n}+ z_{\alpha/2}\frac{\hat{s}_{r,n}}{n} \right]
\end{eqnarray}
and 
\begin{eqnarray}\label{eq: Normal Ratio CI}
	\left[\frac{T^2_{r,n}}{T_{r,n}+z_{\alpha/2} \frac{\hat{s}_{r,n}}{n}}, \frac{T^2_{r,n}}{T_{r,n} -z_{\alpha/2} \frac{\hat{s}_{r,n}}{n}}\right],
\end{eqnarray}
where $z_{\alpha/2}$ denotes the $1-\alpha/2$ quantile of the standard normal distribution. Since $\pi_{r,n}\in[0,1]$, we make the following adjustments. For the CI in \eqref{eq: Normal CI}, we take the lower bound to be $0$ when $T_{r,n}\le z_{\alpha/2}\frac{\hat{s}_{r,n}}{n}$ and the upper bound to be $1$ when $T_{r,n}+ z_{\alpha/2}\frac{\hat{s}_{r,n}}{n}\ge1$. Furthermore, when $\hat{s}_{r,n}=0$, it necessarily follows that $T_{r,n}=0$ and the CI reduces to the point estimator $\{0\}$ in this case. For the CI in \eqref{eq: Normal Ratio CI}, we take it to be the point estimator $\{0\}$ when $T_{r,n}=0$ and we take the upper bound to be $1$ when $0<T_{r,n}\le z_{\alpha/2}\frac{\hat{s}_{r,n}}{n}$. Note that the lower bound cannot be negative in this case. We refer to the CI in \eqref{eq: Normal CI} as the \textbf{Normal CI} and to the one in \eqref{eq: Normal Ratio CI} as the \textbf{Normal Ratio CI}. The Normal CI was first given for $r=0$ in \cite{Zhang:Huang:2008} and it was extended to $r\ge1$ in \cite{Zhang:2013}, see also \cite{Zhang:Zhang:2009} and \cite{Chang:Grabchak:2023}. We have not seen the Normal Ratio CI written out before. However, the theory was developed in \cite{Grabchak:Zhang:2017} and \cite{Chang:Grabchak:2023}.

If $s_{r,n}\to c\in(0,\infty)$, then, under mild conditions (see Appendix \ref{sec: cond}), we have asymptotic Poissonity, i.e., a Poisson distribution in the limit. In this case, for $c^*=c^2/(r+1)^2$, we have 
\begin{eqnarray}\label{eq: asymp Pois}
	\rE\left(\frac{n-r}{r+1}\pi_{r,n}-c^*\right)^2\to0, \mbox{   } \frac{n}{r+1}\pi_{r,n}\conp c^*, \mbox{  and  }\frac{n}{r+1}T_{r,n}\cond \mathrm{Pois}(c^*),
\end{eqnarray}
where $\mathrm{Pois}(c^*)$ denotes a Poisson distribution with mean $c^*$. There are many approaches for deriving a CI for the mean of a Poisson distribution; see, e.g.,\ the survey paper \cite{Sahai:Khurshid:1993} and the references therein. We follow a standard approach based on a chi-squared approximation, which can be found in Section 4.7.3 of \cite{Johnson:Kemp:Kotz:2005}. The resulting $(1-\alpha)100\%$ CI for $\pi_{r,n}$ is given by
\begin{eqnarray}\label{eq: Poisson CI}
	\left[ \frac{r+1}{2n}\chi^2 \left(\alpha/2, \frac{2n T_{r,n}}{r+1}\right), \frac{r+1}{2n} \chi^2 \left(1 - \alpha/2, \frac{2n T_{r,n}}{r+1} + 2\right) \right],
\end{eqnarray}
where $\chi^2(a,\nu)$ denotes the $a$ quantile of the chi-squared distribution with $\nu$ degrees of freedom. When $T_{r,n}=0$, we take this interval to be $\left[0, \right.$ $\left.\frac{r+1}{2n} \chi^2 \left(1 - \alpha/2, 2\right)\right]$. Thus, this CI never degenerates to a point. We refer to the CI in \eqref{eq: Poisson CI} as the \textbf{Poisson CI}. We have not seen this CI written out before. However, the theory was derived in Theorem 2 of \cite{Zhang:Zhang:2009} for $r=0$ and Theorem 2.3 in \cite{Chang:Grabchak:2023} for $r\ge1$. We note that, for $r=0$, a similar CI for the related quantity $1-\pi_{0,n}$ is given in \cite{Esty:1982}.

From a theoretical perspective, the CIs based on asymptotic normality should work better when $s_{r,n}\to\infty$ and the one based on asymptotic Poissonity should work better when $s_{r,n}\to c\in(0,\infty)$. In practice, of course, we do not know the asymptotics of $s_{r,n}$. However, we can estimate $s_{r,n}$ by $\hat s_{r,n}$, which suggests the following heuristic for choosing between the two types of CIs. First, select a threshold $V>0$. If $\hat s_{r,n}< V$, we use the Poisson CI given in \eqref{eq: Poisson CI}, otherwise we use the Normal CI given in \eqref{eq: Normal CI}. We refer to this as the \textbf{Heuristic CI}. Here, $V$ is a tuning parameter. In principle, one could substitute the Normal Ratio CI given in \eqref{eq: Normal Ratio CI} for the Normal CI. However,  we opt not to do so as we found that the resulting CI does not, in general, perform as well.

We now discuss several variants of the above CIs that appear in the literature. Perhaps the most prominent is the CI introduced in \cite{Esty:1983}. While there it was only given for $r=0$, it can be formulated for any $r$ as follows. A $(1-\alpha)100\%$ CI for $\pi_{r,n}$ is given by
\begin{eqnarray}\label{eq: Normal Esty CI}
\left[T_{r,n}-\frac{z_{\alpha/2}}{n}\sqrt{\hat{s}_{r,n}^2 -N_{r+1,n}^2/n} , T_{r,n}+ \frac{z_{\alpha/2}}{n} \sqrt{\hat{s}_{r,n}^2 -N_{r+1,n}^2/n}\ \right].
\end{eqnarray}
While asymptotically it is equivalent to the CI in \eqref{eq: Normal CI}, for any finite $n$ this CI is narrower. In a different direction, all of the CIs discussed above can be constructed using a modified version of Turing's estimator. This version is used in, e.g., \cite{Zhang:2013} and \cite{Grabchak:Zhang:2017} and it is given by
\begin{eqnarray}\label{eq: modified Turing}
T^*_{r,n}=\frac{N_{r+1,n}}{n-r}(r+1)=\frac{n}{n-r}T_{r,n}.
\end{eqnarray}
The bias of this estimator is as in \eqref{eq: bias}, but with $p_{\ell}$ in place of $(p_{\ell}-r/n)$. Asymptotically there is no difference between the two. All of the CIs discussed above can be modified to use $T^*_{r,n}$ in place of $T_{r,n}$. We note that $T_{r,n}$ is the form that was originally introduced in \cite{Good:1953}.

\section{Simulations}\label{sec: sims}

We performed a series of simulations to better understand the performance of the various CIs discussed in Section \ref{sec: CIs}. In the interest of space, we only give detailed results for three CIs: the Normal CI given in \eqref{eq: Normal CI}, the Poisson CI given in \eqref{eq: Poisson CI}, and a version of the Heuristic CI with tuning parameter $V=2$. Recall that the Heuristic CI chooses the Poisson CI when $\hat s_{r,n}< V$ and otherwise it chooses the Normal CI. We tried a variety of choices for $V$ and $V=2$ seemed to work very well. We leave the question of how to optimize the choice of $V$ for future work.

While we do not present detailed results for the other CIs, we performed many simulations with them and now give a brief summary. We found that there is, essentially, no difference in using CIs based on the modified Turing's estimator $T^*_{r,n}$ in place of $T_{r,n}$, even for small sample sizes. When comparing the CI in \eqref{eq: Normal Esty CI} with the Normal CI, we observed that the Normal CI was typically a little wider, but that it had better coverage, sometimes significantly so. The Normal Ratio CI given in \eqref{eq: Normal Ratio CI} typically had comparable coverage to the Normal CI, but it tended to have wider, sometimes much wider, CIs. This seems to be caused by the fact that the upper bound in this CI can be unstable due to the fact that the denominator can approach $0$. We now present our detailed simulation results, which focus on the Normal CI, the Poisson CI, and the Heuristic CI in the context of three families of distributions.

\subsection*{Discrete Uniform Distribution}

\begin{figure}[t]
	\centering
	\includegraphics[width=\linewidth]{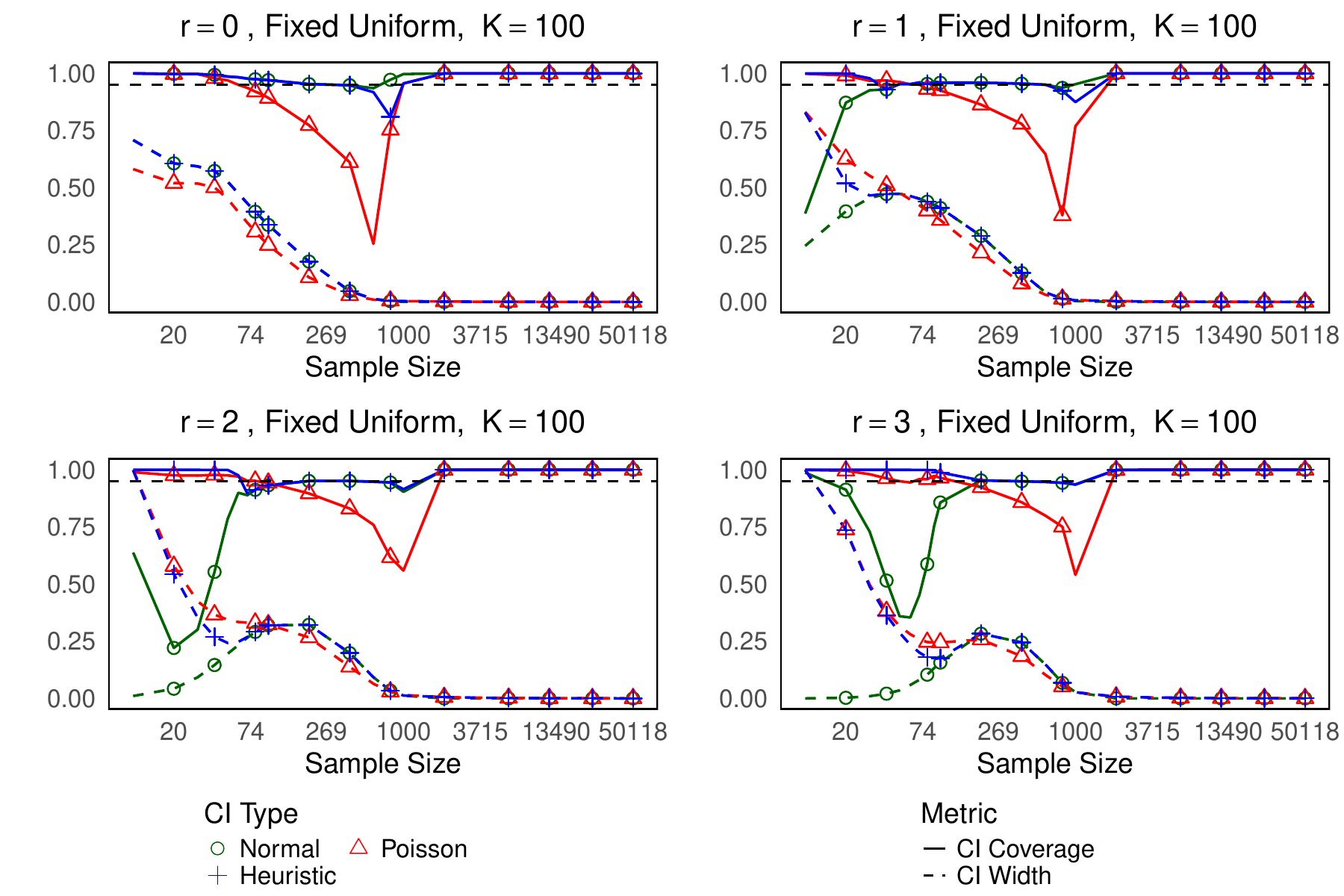}
	\caption[Results for the Discrete Uniform Distribution.]{
		\textbf{(a)} $K = 100$.
		Results for the Discrete Uniform Distribution. Each panel shows the coverage proportions and the mean widths for the three $95\%$ confidence intervals for different choices of $r$ and $K$. These results are based on $N = 5000$ replications. The sample size on the $x$-axis is presented on a log (base 10) scale. The horizontal dashed line indicates the nominal level of $0.95$.
	}
	\label{fig:fixed_uniform}
\end{figure}

\begin{figure}[t]\ContinuedFloat
	\centering
	\includegraphics[width=\linewidth]{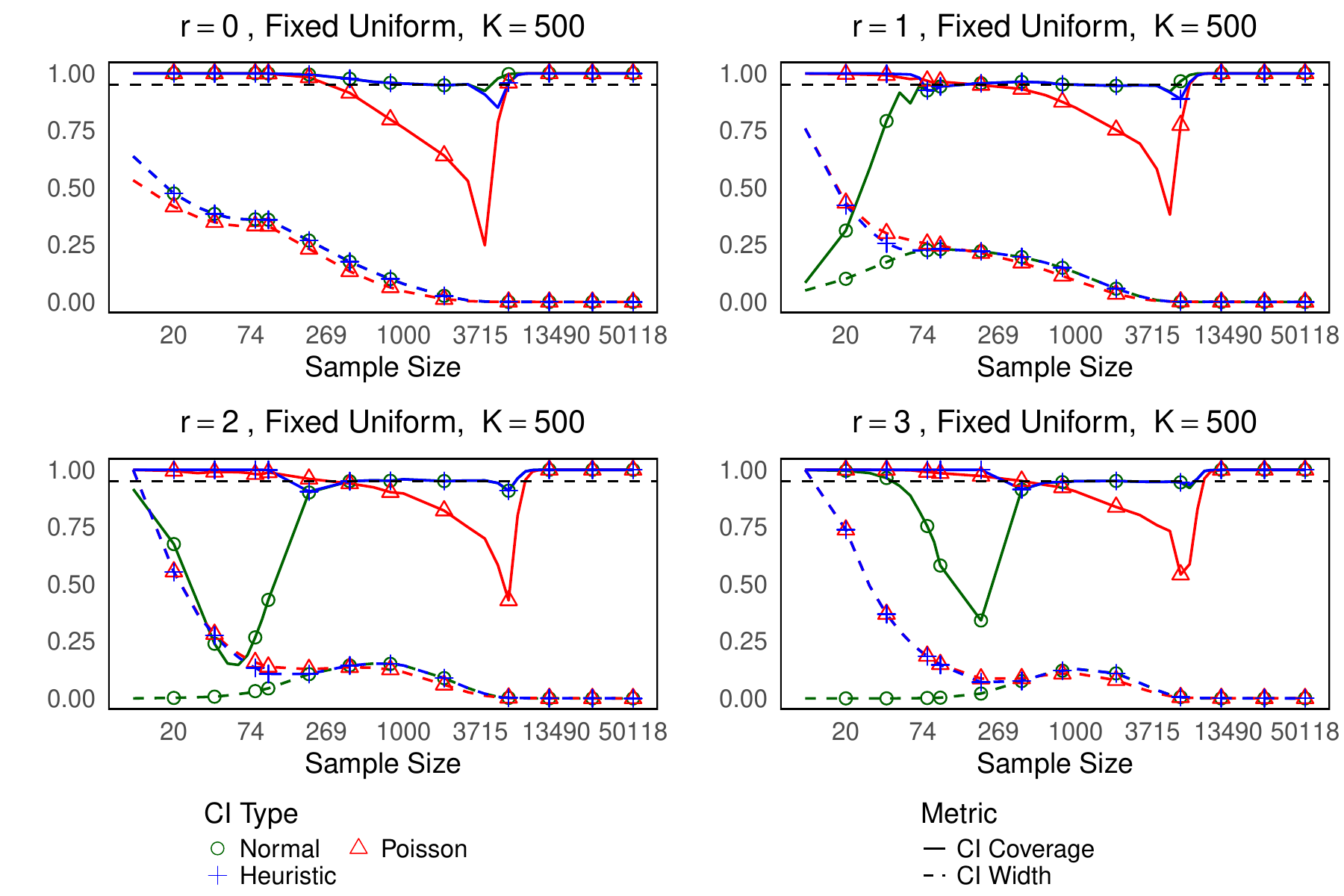}
	\caption[]{\textbf{(b)} $K = 500$. Results for the Discrete Uniform, continued.}
\end{figure}

\begin{figure}[t]\ContinuedFloat
	\centering
	\includegraphics[width=\linewidth]{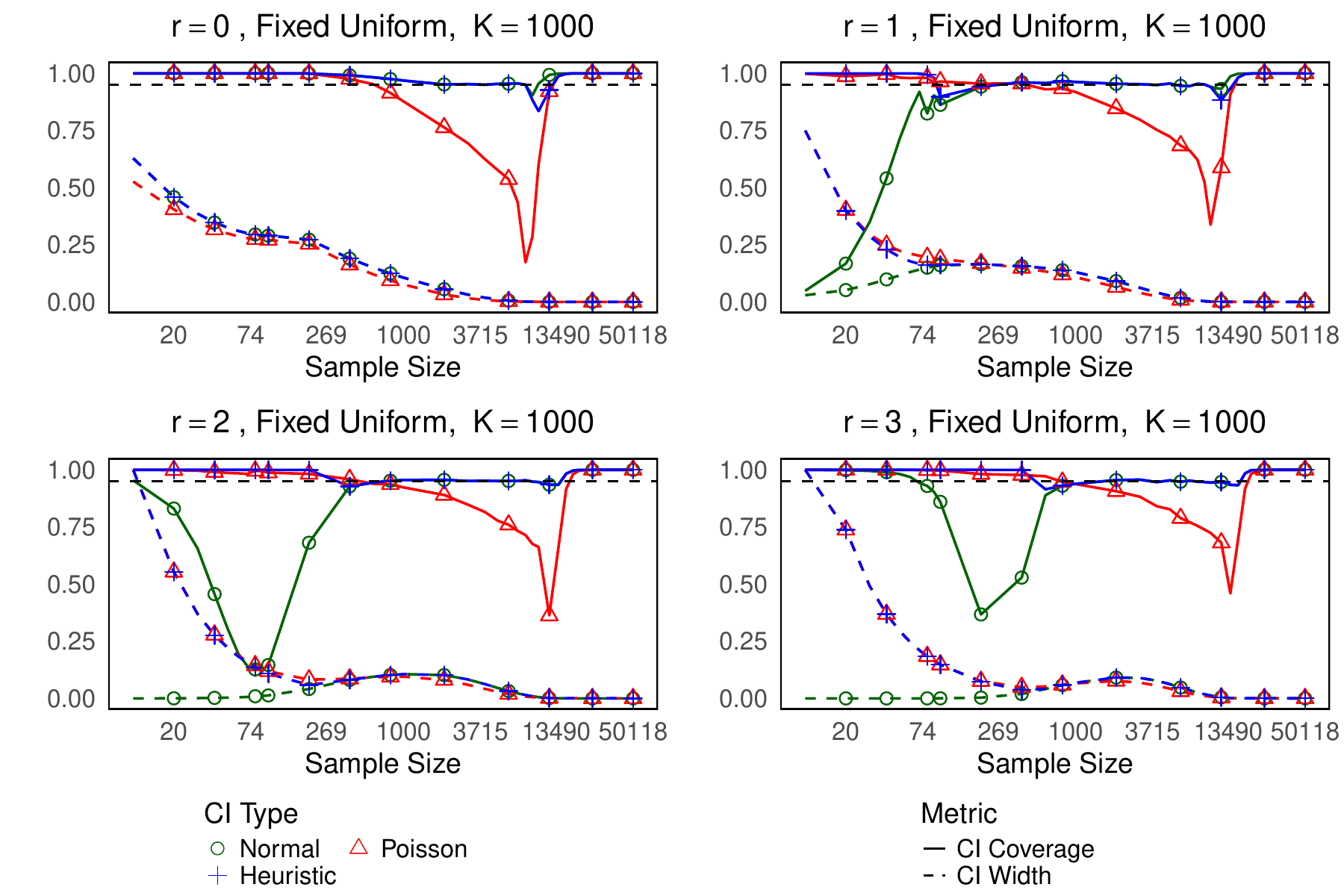}
	\caption[]{\textbf{(c)} $K = 1000$. Results for the Discrete Uniform, continued.}
\end{figure}

The discrete uniform distribution has a probability mass function (pmf) given by
\begin{equation}\label{eq: fixed unif}
	p_\ell= 1/K, \ \ \ell = 1, 2, 3, \dots, K,
\end{equation}
where integer $K\ge1$ is a parameter. It is easily checked that 
$$s^2_{r,n} = K\left(r+1+n/K\right)e^{-n/K} \frac{(n/K)^{r+1}}{r!}\to0$$ as $n\to\infty$. Thus, the known sufficient conditions needed for the asymptotics underpinning the various CIs do not hold. Nevertheless, our simulation results suggest that the CIs generally work well.  

For the simulations, we considered three choices of $K=100, 500,1000$ and a variety of sample sizes $n$. For each choice of $K$ and $n$, we simulated $N=5000$ samples (replications). For each sample and each $r=0,1,2,3$, we constructed the Normal CI, the Poisson CI, and the Heuristic CI, each at the $95\%$ confidence level. We then found the coverage proportion, i.e.,\ the proportion of the time that each CI contained the true value $\pi_{r,n}$, and we calculated the sample means of the widths of the CIs. This information is summarized in Figure~\ref{fig:fixed_uniform}(a)--(c). We can see that the Normal CI and the Heuristic CI have very good performance, with the Poisson CI typically performing worse. While there are some cases where the Poisson CI outperforms the Normal CI in terms of coverage, this is at the expense of wider, often significantly wider, intervals. The Heuristic CI always has the best performance in terms of coverage and its width is never wider than the largest width of the other two.

It may be interesting to note that the proportion of CIs containing the true value does not converge to $0.95$, but to $1$. This cannot be explained through having wide CIs, as the widths are very close to $0$ when this happens. The explanation is that the discrete uniform distribution has only a finite number of possible values and, for large enough $n$, we are likely to  see all letters multiple times. This would imply that $\pi_{r,n}=0$, $T_{r,n}=0$, and $\hat s_{r,n}=0$. In this case the Normal CI becomes the singleton $\{0\}$ and any interval that contains $0$ will contain $\pi_{r,n}$. A similar situation holds for the Poisson CI, although it never fully degenerates into a point.

What may be even more interesting is that the coverage of the Normal CI appears to converge to $0.95$ for large, but not too large, values of the sample size $n$. Such apparent convergence before the actual convergence is sometimes called pre-limit behavior, see \cite{Grabchak:Samorodnitsky:2010} and the references therein for a discussion in a somewhat different context. The theoretical explanation for this pre-limiting, apparent convergence is the fact that, while the limits in \eqref{eq: asymp normal}, \eqref{eq: asymp ratio normal}, and \eqref{eq: asymp Pois} do not hold for any fixed value of $K$, they can hold when $K$ increases with the sample size. Specifically, assume that, for sample size $n$, the parameter $K$ is given by $K=\lfloor n^\gamma\rfloor$, where $\gamma>0$ and $\lfloor \cdot\rfloor$ is the floor function. In this case, we have a sequence of discrete uniform distributions, where the $n$th distribution has pmf
\begin{eqnarray}\label{eq: dynamic unif}
	p_{\ell}^{(n)} = \frac{1}{\lfloor n^\gamma\rfloor}, \ \ \ell=1,2,\dots,\lfloor n^\gamma\rfloor,
\end{eqnarray}
for some $\gamma>0$ that does not depend on $n$. This distribution is changing dynamically with the sample size, and we refer to it as the dynamic discrete uniform. In this context, we refer to the distribution given in \eqref{eq: fixed unif} as the fixed discrete uniform distribution. Note that the larger the value of $\gamma$, the faster the number of letters in the distribution grows. When discussing dynamic distributions, we take $s_{r,n}$ to be as given by \eqref{eq: asymptotic sd}, but with $p_{\ell}^{(n)}$ in place of $p_\ell$.

We derive asymptotic theory for the dynamic discrete uniform in Appendix \ref{sec: dynamic uniform theory}. The results can be summarized as follows:
\begin{itemize}
	\item[1.] For $r=0$ and $\gamma=1$ and for $r\ge1$, $\gamma>1$, and $\gamma<1+1/r$, we have asymptotic normality.
	\item[2.] For $r\ge1$, $\gamma>1$, and $\gamma=1+1/r$, we have asymptotic Poissonity with mean $\frac{1}{(r+1)!}$.
	\item[3.] For $r\ge1$, $\gamma>1$, and $\gamma>1+1/r$, we have $s_{r,n}\to0$ with $s^2_{r,n}\sim \frac{r+1}{r!} n^{r+1-\gamma r}$. 
	\item[4.] For $\gamma\in (0,1)$, we have $s_{r,n}\to0$ with $s^2_{r,n}\sim n^{2+r-\gamma(1+r)}e^{-n/\lfloor n^\gamma\rfloor}$.
	\item[5.] Current state-of-the-art asymptotic results cannot tell us what happens when $r=0$ and $\gamma>1$ or when $r\ge1$ and $\gamma=1$. However, we do know that $s_{r,n}\to\infty$ in these cases.
\end{itemize}
Recall that $s_{r,n}$ is the asymptotic standard deviation of $n\left(T_{r,n}-\pi_{r,n}\right)$. Thus, $s_{r,n}\to0$ means that there is either no asymptotic distribution or that the rate of convergence is strictly slower than $1/n$. For $r\ge1$ we have $s_{r,n}\to0$ both when $\gamma$ is small ($\gamma\in (0,1)$) and when it is large ($\gamma>1+1/r$). This seems to be for opposite reasons. When $\gamma$ is small, the number of letters grows slowly as compared to the sample size, and for large samples we are likely to have observed all of the letters more than $r$ times. On the other hand, when $\gamma$ is large, the number of letters grows quickly, and for large samples we are unlikely to have seen any letter $r$ or more times.

In the context of the fixed discrete uniform distribution, the asymptotic results suggest the following. Assume that the parameter $K$ in \eqref{eq: fixed unif} is relatively large. If for some values of $r$ and $\gamma$ we have asymptotic normality in the dynamic case, then for (large) sample sizes $n$ that are on the order of $K^{1/\gamma}$, we should be close to normal and the Normal CI should work well. However, for sample sizes $n$ that are much larger than $K^{1/\gamma}$, we should no longer be close to normal. A similar result should hold for asymptotic Poissonity. Note that for $r\ge1$, asymptotic Poissonity holds for larger values of $\gamma$ (smaller values of $1/\gamma$) than the values where asymptotic normality holds. Thus, we would expect the Poisson CI to work well for large (but not too large) sample sizes, then for even larger sample sizes we would expect the Normal CI to work well, and for very large sample sizes everything should approach $1$ by the arguments given previously. This is very much in keeping with what we see in Figure~\ref{fig:fixed_uniform}(a)--(c).

\begin{figure}[t]
	\centering
	\includegraphics[width=\linewidth]{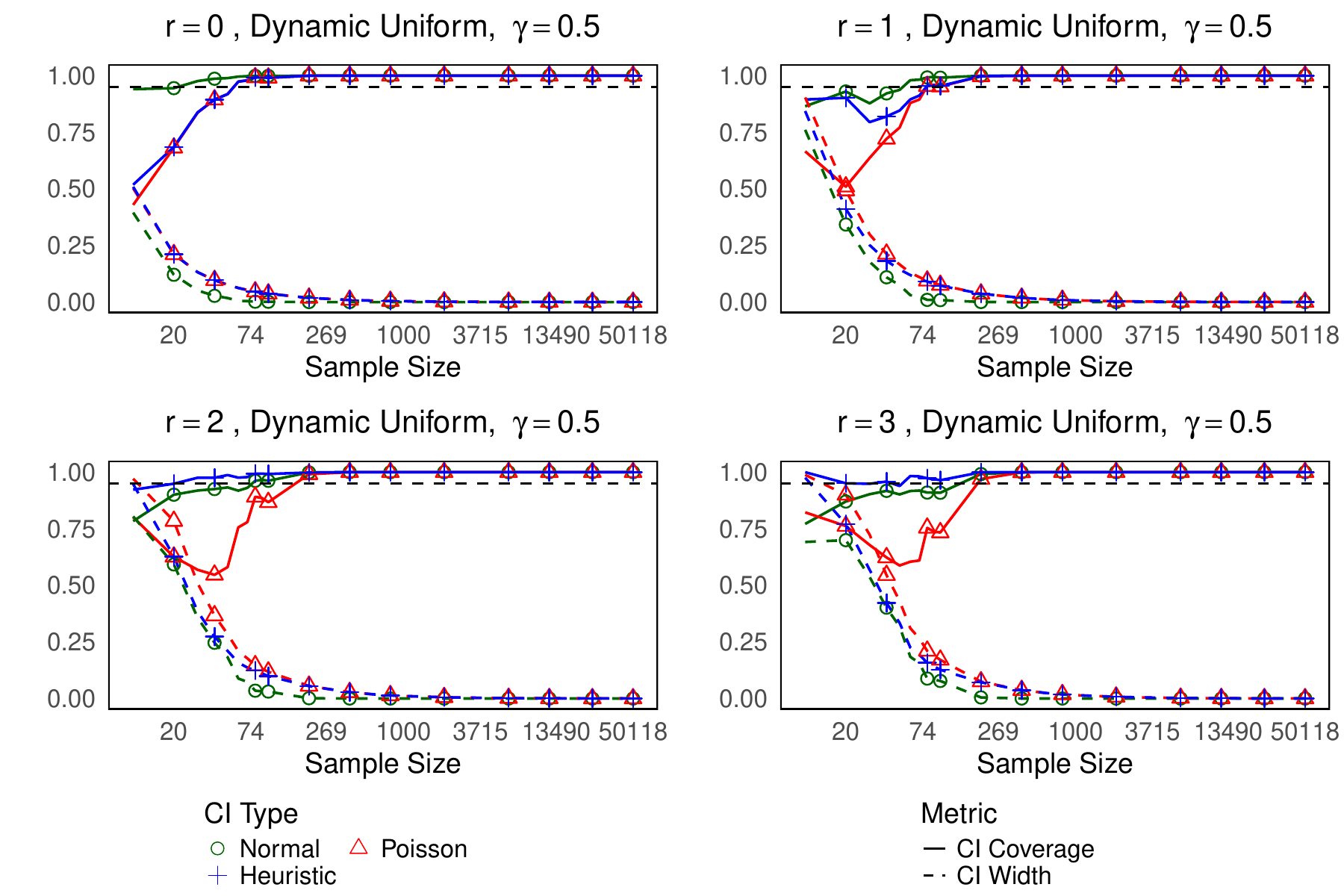}
	\caption[Results for the Dynamic Discrete Uniform.]{
		\textbf{(a)} $\gamma = 0.5$. 
		Results for the Dynamic Discrete Uniform. Each panel shows the coverage proportions and the mean widths for the three $95\%$ confidence intervals for different choices of $r$ and $\gamma$. 
		These results are based on $N = 5000$ replications. 
		The sample size on the $x$-axis is presented on a log (base 10) scale. 
		The horizontal dashed line indicates the nominal level of $0.95$.}
	\label{fig:dynamic_uniform}
\end{figure}

\begin{figure}[t]\ContinuedFloat
	\centering
	\includegraphics[width=\linewidth]{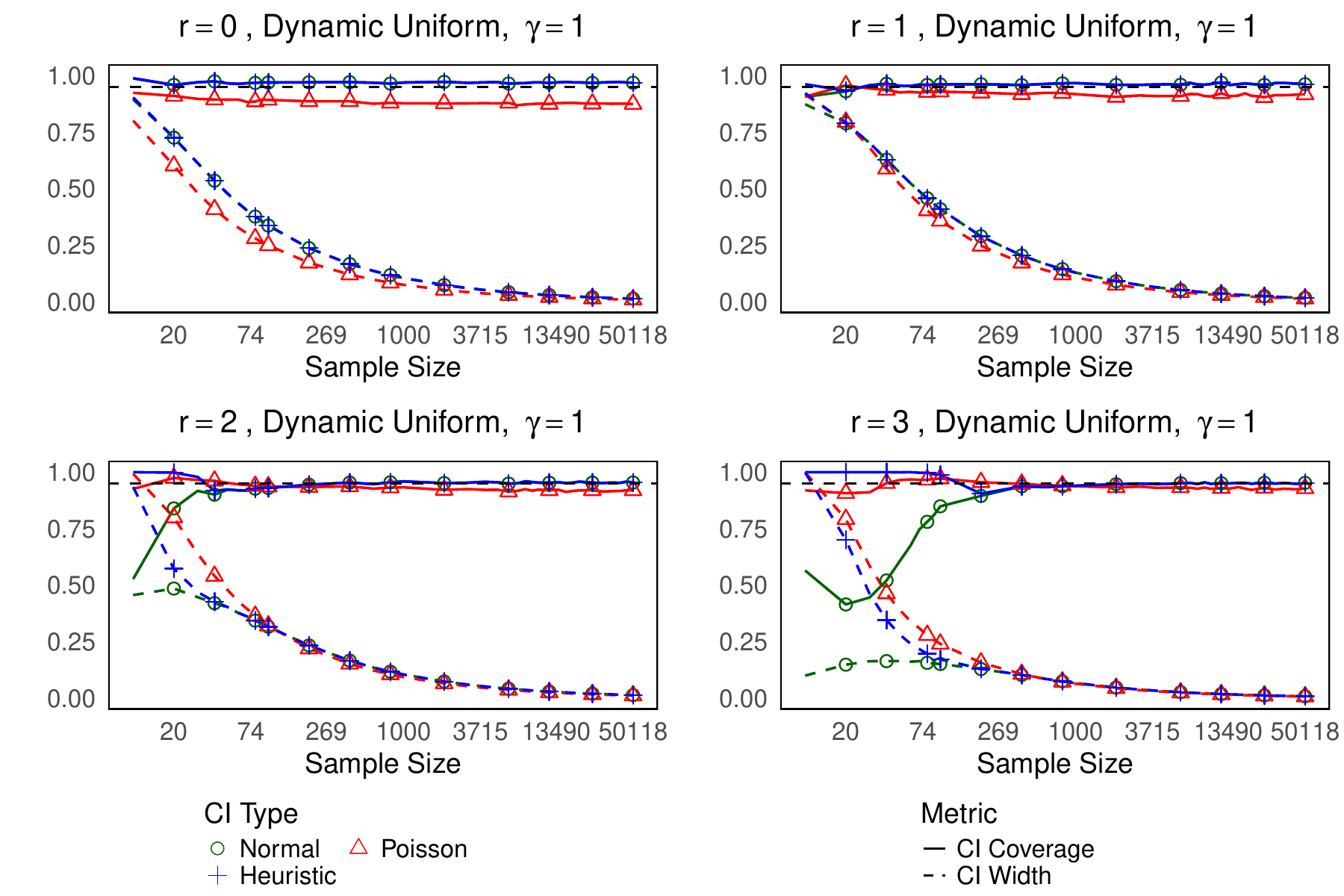}
	\caption[]{\textbf{(b)} $\gamma = 1$. Results for the Dynamic Discrete Uniform, continued.}
\end{figure}

\begin{figure}[t]\ContinuedFloat
	\centering
	\includegraphics[width=\linewidth]{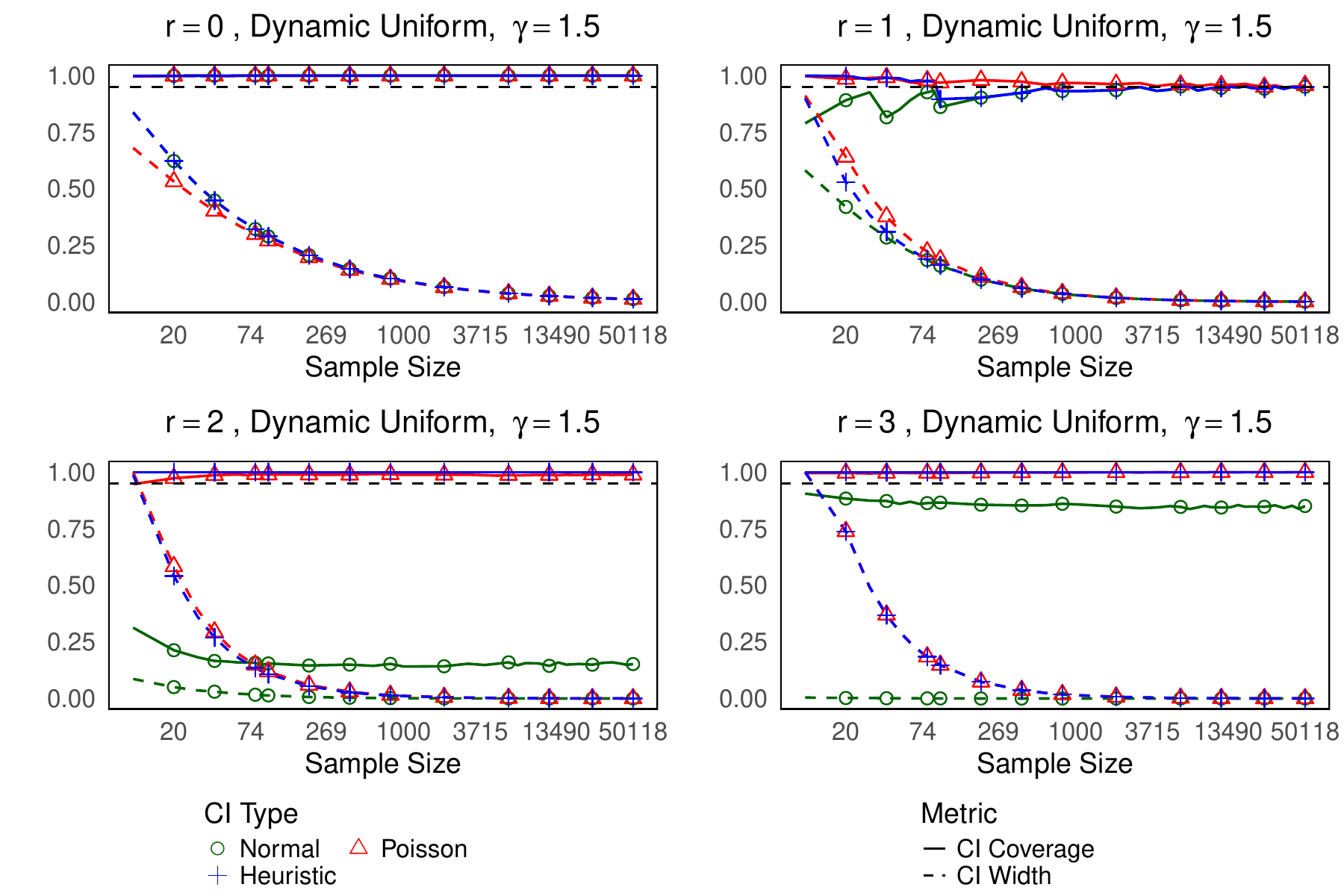}
	\caption[]{\textbf{(c)} $\gamma = 1.5$. Results for the Dynamic Discrete Uniform, continued.}
\end{figure}

To better understand the dynamic situation, we performed a simulation study for this case. The results are given in Figure~\ref{fig:dynamic_uniform}(a)--(c). The simulations were performed in a manner similar to the fixed discrete uniform case, except that now, instead of fixing $K$, we fix $\gamma$ and then choose $K$ dynamically depending on $n$. We can see that the Heuristic CI again has the best overall performance in terms of coverage. However, its performance is slightly worse than the Normal CI for small sample sizes when $\gamma=0.5$ and $r=0,1$.

Based on the theoretical results, we have asymptotic normality when $r=0$, $\gamma=1$ and when $r=1$, $\gamma=1.5$. In these cases, we can see that, for larger sample sizes, the Normal CI is close to the $0.95$ line. It should be noted that the Poisson CI also performs well here. This may be related to the fact that the normal distribution provides a good approximation to the Poisson when the mean is large, but this requires further study. We have asymptotic Poissonity when $r=2$ and $\gamma=1.5$. We can see that the Poisson CI works well in this case, while the Normal CI has very poor performance. In this case, the mean of the asymptotic Poisson distribution is $1/6$ and the normal approximation to the Poisson will not work well.

We now turn to the situations where $s_{r,n}\to 0$. For $\gamma=0.5$ we can observe that all three CIs converge to $1$ and that the widths get very close to $0$. This is in keeping with what we saw for the fixed discrete uniform and seems to be for similar reasons. The situation is very different when $r=3$ and $\gamma=1.5$. Here, the Poisson CI has very high coverage, but the Normal CI does not. However, we conjecture that this situation should actually be similar to that of $\gamma=0.5$, but that we would need much larger sample sizes to see this. Our justification is that, in this case, $s_{r,n}\to 0$ at a much slower rate. To check this conjecture we would need to perform simulations with much larger sample sizes. Due to computational constraints, we are only able to obtain preliminary results. For sample sizes $n=10^7$ and $n=10^8$, we simulated $100$ replications. The results are summarized in Table \ref{table: Table}. We can see that the coverage of the Normal CI does seem to approach $1$, but this requires further study. We note that, in all $200$ of these simulations, we observed $T_{3,n}=0$. 

\begin{table}
	\centering
		\begin{tabular*}{\textwidth}{@{\extracolsep{\fill}} cccc }
			\hline
			CI & $n$ & Coverage &  Mean Width\\
			\hline
			Normal& $10^7$ & 82/100  & $0$ \\ 
			Poisson & $10^7$ & 100/100& $1.476\times10^{-6}$ \\ 
			Heuristic & $10^7$ & 100/100& $1.476\times10^{-6}$ \\ 
			Normal& $10^8$ & 93/100  & 0 \\ 
			Poisson & $10^8$ & 100/100&  $1.476\times10^{-7}$\\
			Heuristic & $10^8$ & 100/100&  $1.476\times10^{-7}$\\
			\hline
		\end{tabular*}
	\vspace{4pt}
	\caption{Results for Dynamic Uniform with $\gamma=1.5$ and $r=3$ for large sample sizes $n$. These are based on $N=100$ replications.}
	\label{table: Table}
\end{table}

We conclude this discussion by considering the situations for which theoretical results are not known. For $r=0$ and $\gamma=1.5$, all three CIs seem to perform well, and they essentially give $100\%$ coverage, although with fairly wide CIs. These results do not suggest that asymptotic normality or asymptotic Poissonity hold, but they do not preclude it either. On the other hand, when $\gamma=1$ and $r=1,2,3$, the situations look very much like what we would expect if asymptotic normality held. We leave the problem of verifying this theoretically as an important direction for future work.

\subsection*{Geometric Distribution}

\begin{figure}[t]
	\centering
	\includegraphics[width=\linewidth]{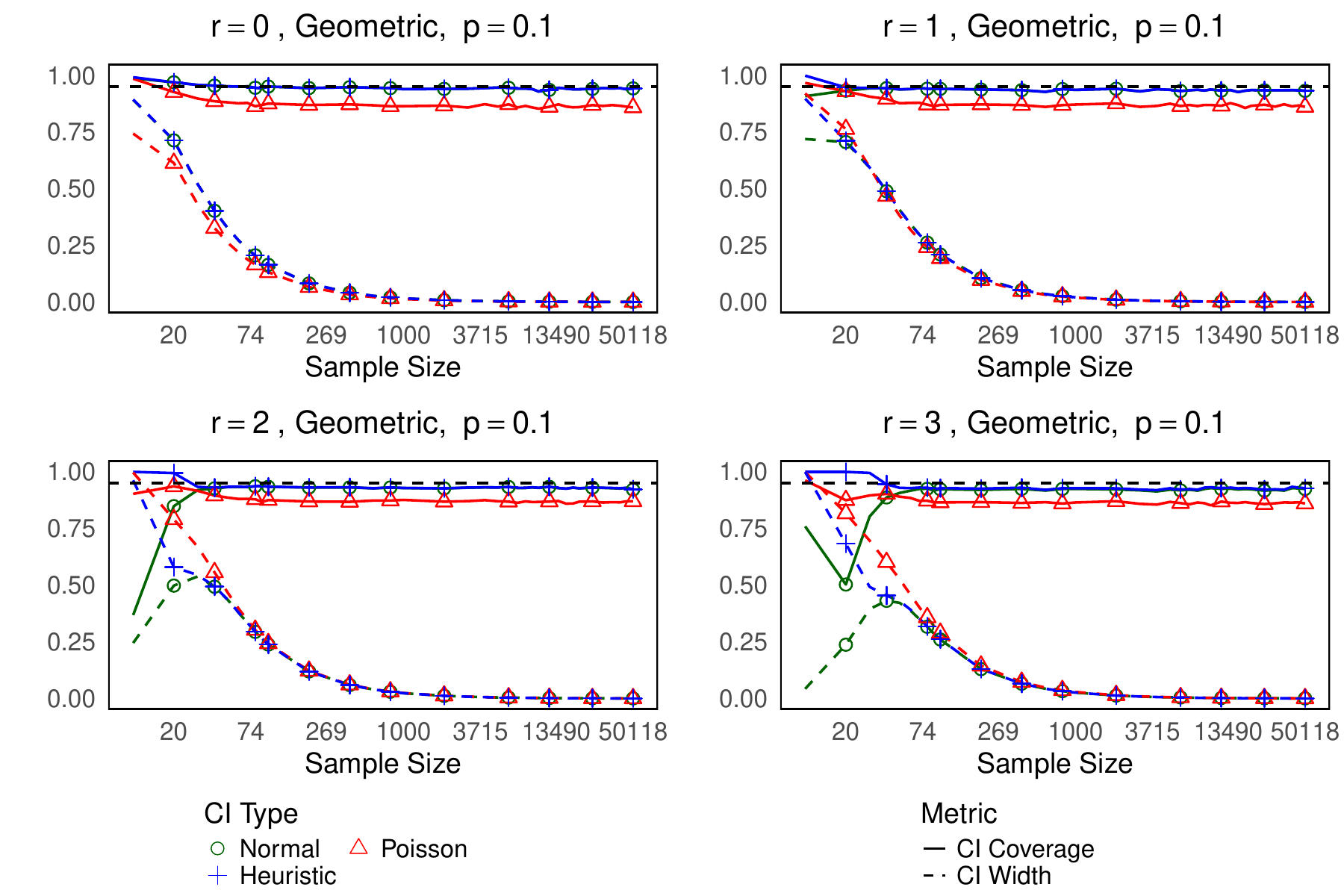}
	\caption[Results for the Geometric Distribution.]{
		\textbf{(a)} $p = 0.1$. 
		Results for the Geometric Distribution. Each plot gives the coverage proportions and the mean widths for the three $95\%$ CIs. The plots are based on $N=5000$ replications. The sample size on the $x$-axis is presented on a log (base 10) scale. The horizontal dashed line indicates the nominal level of $0.95$.}
	\label{fig:geometric}
\end{figure}

\begin{figure}[t]\ContinuedFloat
	\centering
	\includegraphics[width=\linewidth]{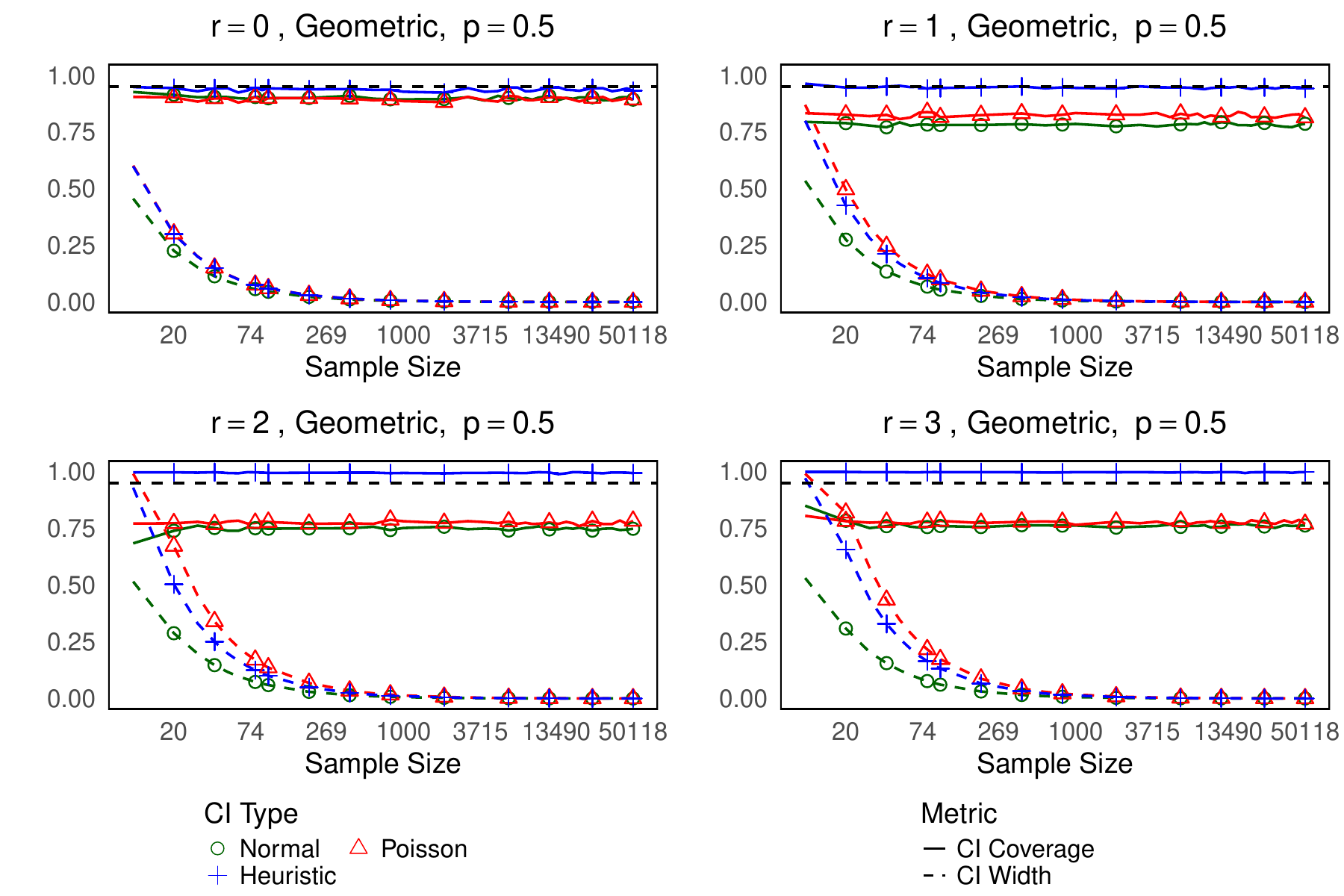}
	\caption[]{\textbf{(b)} $p = 0.5$. Results for the Geometric Distribution, continued.}
\end{figure}

\begin{figure}[t]\ContinuedFloat
	\centering
	\includegraphics[width=\linewidth]{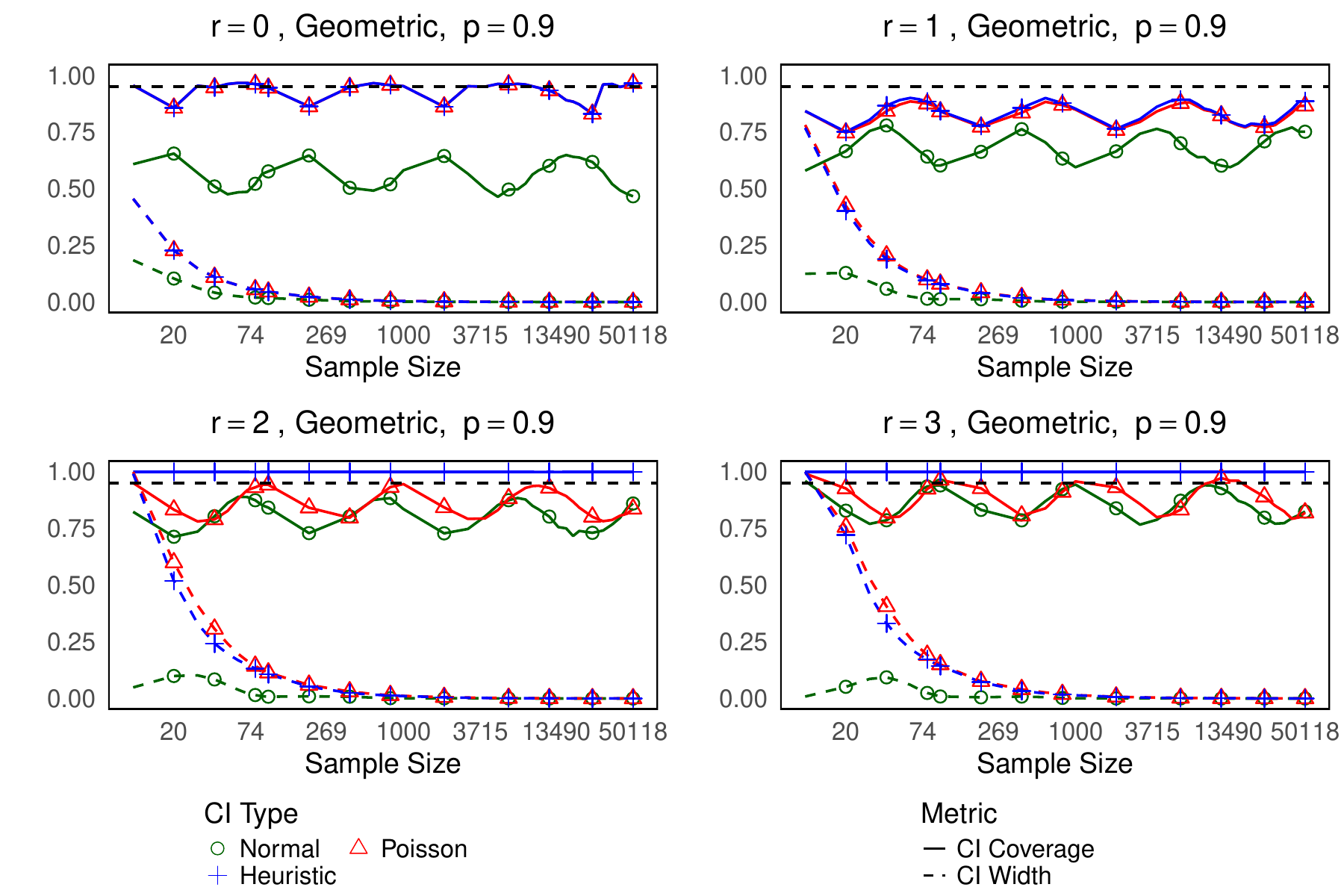}
	\caption[]{\textbf{(c)} $p = 0.9$. Results for the Geometric Distribution, continued.}
\end{figure}

\begin{figure}[t]\ContinuedFloat
	\centering
	\includegraphics[width=\linewidth]{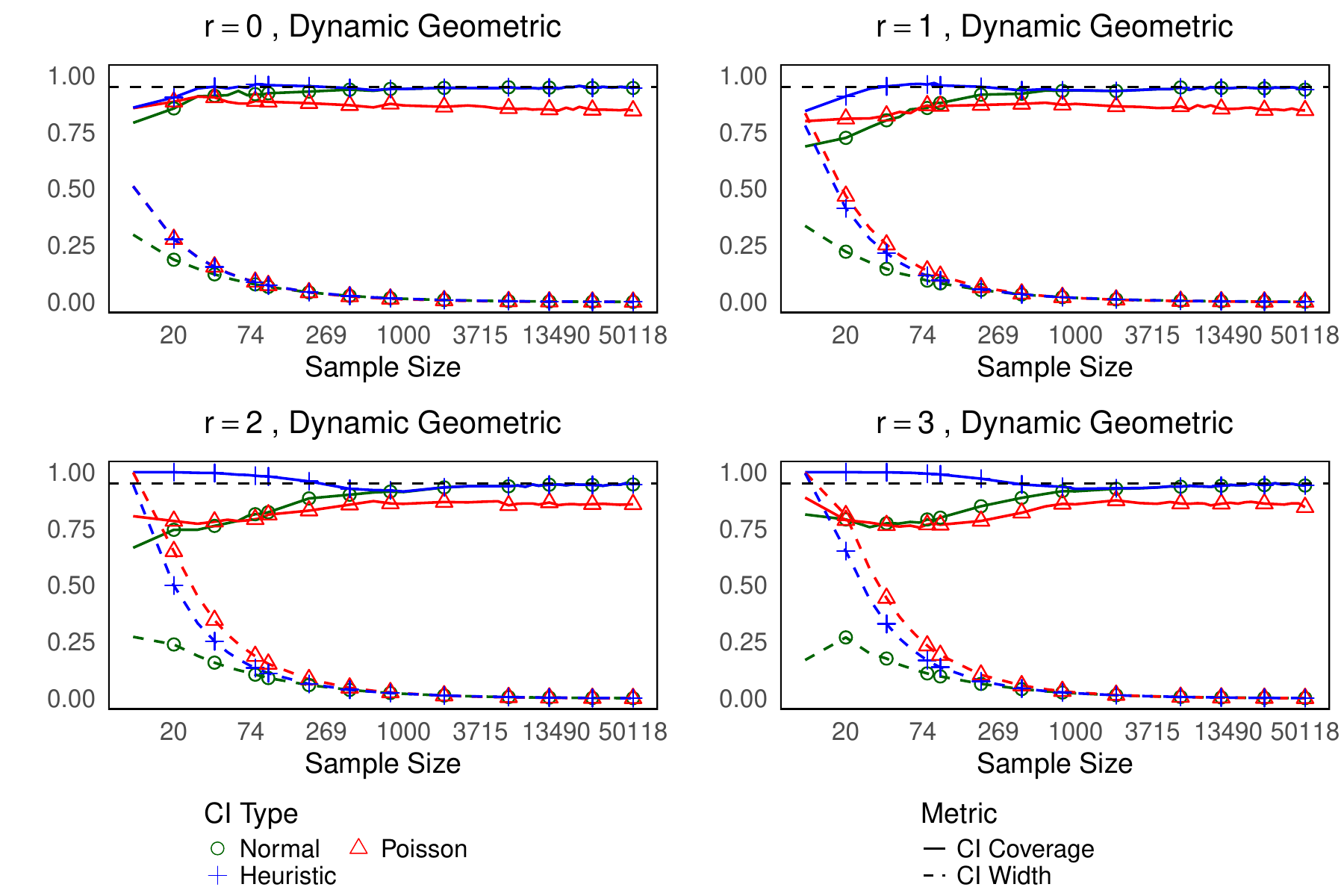}
	\caption[]{\textbf{(d)} Dynamic Geometric distribution. Results for the Geometric Distribution, continued.}
\end{figure}

The geometric distribution has a pmf given by
\begin{equation}\label{eq: fixed geo}
	p_\ell = p(1-p)^{\ell-1}, \ \ \ell = 1, 2, 3, \dots,
\end{equation}
where $p \in (0,1)$ is a parameter. The following is an immediate consequence of the slightly more general Proposition \ref{prop: norm for dyn geo} given below.

\begin{prop}
	Let $a=-1/\log(1-p)$. We have
	\begin{eqnarray}
		&& 2\max\{0,a(r+1) -(r+1)^{r+1}e^{-r-1} - (r+2)^{r+2}e^{-r-2}\}\nonumber \\
		&&\qquad \le \liminf_{n\to\infty} s^2_{r,n}
		\le \limsup_{n\to\infty} s^2_{r,n}\nonumber\\
		&&\qquad \le 2a(r+1) +2(r+1)^{r+1}e^{-r-1} + 2(r+2)^{r+2}e^{-r-2}. \label{eq: bounds for fixed geo}
	\end{eqnarray}
\end{prop}

This means that the known sufficient conditions for \eqref{eq: asymp normal} and \eqref{eq: asymp ratio normal} do not hold. It is not known if $\lim_{n\to\infty} s^2_{r,n}$ exists, although known results about a related quantity lead us to conjecture that it does not, see Theorem 3 in \cite{Zhang:2018}. Even if it does exist, the known sufficient conditions for \eqref{eq: asymp Pois} do not hold, see the discussion just below Theorem 2.3 in \cite{Chang:Grabchak:2023}.

The simulations were performed in a similar manner to those for the discrete uniform distribution. We considered the parameter values: $p=0.1, 0.5, 0.9$, and the results are summarized in Figures~\ref{fig:geometric}(a)--(c). We can see that the widths converge to $0$, but that the coverage proportions do not approach $0.95$. Instead, they appear to be either constant or, in the case of $p=0.9$, they seem to oscillate around a fixed value. It may be interesting to note that, in the situations where they oscillate, they seem to do so with a period that appears constant on a log scale. We do not know why this is, but it suggests an interesting direction for future work.

In all cases, the Heuristic CI seems to work very well. We can see two situations. In the first, the Heuristic CI pretty much always selects one method. This happens when $p=0.1$ with $r=0,1,2,3$ and when $p=0.9$ with $r=0,1$. In all of these cases, the method that is selected is the best performing one. It should be noted that this best method is not always the same one. Specifically, it is the Normal CI when $p=0.1$ and the Poisson CI when $p=0.9$.

The other situation, which happens in all of the remaining cases, sees the Heuristic CI significantly outperform the other two. For instance, when $r=2$ and $p=0.5$, the Normal and Poisson CIs each have about $75\%$ coverage, while the Heuristic CI is close to $100\%$. Clearly, in these cases, the Heuristic CI alternates between the other two, selecting the better CI for the given simulation.

Even though we do not have asymptotic normality or asymptotic Poissonity, we can see that for $p=0.1$, the Normal CI seems to have coverage close to $0.95$. This can, again, be explained by considering a dynamic situation, where the distribution changes with the sample size. Specifically, assume that the parameter $p=p_n$ depends on the sample size $n$. We will show that we get asymptotic normality if $p_n\to0$ at an appropriate rate. This explains why the Normal CI works well when $p$ is small.

To define the dynamic geometric distribution, let $\{p^{(n)}\}$ be a sequence of real numbers with $0<p^{(n)}<1$, let $a_n = -1/\log(1-p^{(n)})$, and note that $p^{(n)} = 1-e^{-1/a_n}$. We consider the sequence of geometric distributions, where the $n$th distribution has pmf
$$
p_{\ell}^{(n)}= p^{(n)}(1-p^{(n)})^{\ell-1}=\left(e^{1/a_n}-1\right)e^{-\ell/a_n}\ \ \ell=1,2,\dots.
$$
The following result is proved in Appendix \ref{eq: proof prop geo}.

\begin{prop}\label{prop: norm for dyn geo}
	1. If $a_n \rightarrow \infty$ and $a_n/n \rightarrow 0$, then both \eqref{eq: asymp normal} and \eqref{eq: asymp ratio normal} hold.\\
	2. If $a_n\to a\in[0,\infty)$, then \eqref{eq: bounds for fixed geo} holds.
\end{prop}

Simulations for the dynamic geometric distribution are given in Figure~\ref{fig:geometric}(d). Here we take $a_n = \sqrt{n} / 4$, which guarantees asymptotic normality. We can clearly see this illustrated in the plots.

\subsection*{Discrete Pareto Distribution}

We begin by recalling the continuous Pareto distribution. While there are several parametrizations, we consider the one whose probability density function (pdf) is given by
$$
f(x) = \alpha x^{-\alpha-1}, \ \ x>1,
$$
where $\alpha>0$ is a parameter. We can simulate from this distribution by taking $Y=U^{-1/\alpha}$, where $U$ has a (continuous) uniform distribution on $(0,1)$. We say that a random variable $X$ has a Discrete Pareto Distribution if $X\eqd\lfloor Y\rfloor$, where $\eqd$ denotes equality in distribution. In this case, the pmf is given by
$$
p_\ell =  \ell^{-\alpha} - \left(\ell+1 \right)^{-\alpha}, \ \ \ell = 1, 2, 3, \dots.
$$
It is readily checked that 
$$
p_\ell \sim \alpha \ell^{-\alpha-1} \mbox{ as } \ell\to\infty.
$$
From here, Proposition 4.1 in \cite{Chang:Grabchak:2023} implies that \eqref{eq: asymp normal} and \eqref{eq: asymp ratio normal} hold for each $\alpha>0$. Thus, we always have asymptotic normality, and there is no need to consider a dynamic situation in this case.

Our simulation results are presented in Figures~\ref{fig:pareto}(a)--(c). We considered three choices for the parameter $\alpha=0.5, 1.5, 2$. From the plots we can see that the smaller the value of $\alpha$, the faster the coverage proportion of the Normal CI converges to $0.95$. This is in keeping with other results that suggest that Turing's estimator converges quicker for smaller values of $\alpha$, see \cite{Grabchak:Cosme:2017}. We note that the Heuristic CI has very good performance. Once we are in a regime where the Normal CI is close to convergence, the Heuristic CI follows it closely. On the other hand, prior to convergence, the Heuristic CI outperforms the Normal CI in terms of the coverage proportion, sometimes significantly so. However, this is at the expense of having wider CIs.

\begin{figure}[t]
	\centering
	\includegraphics[width=\linewidth]{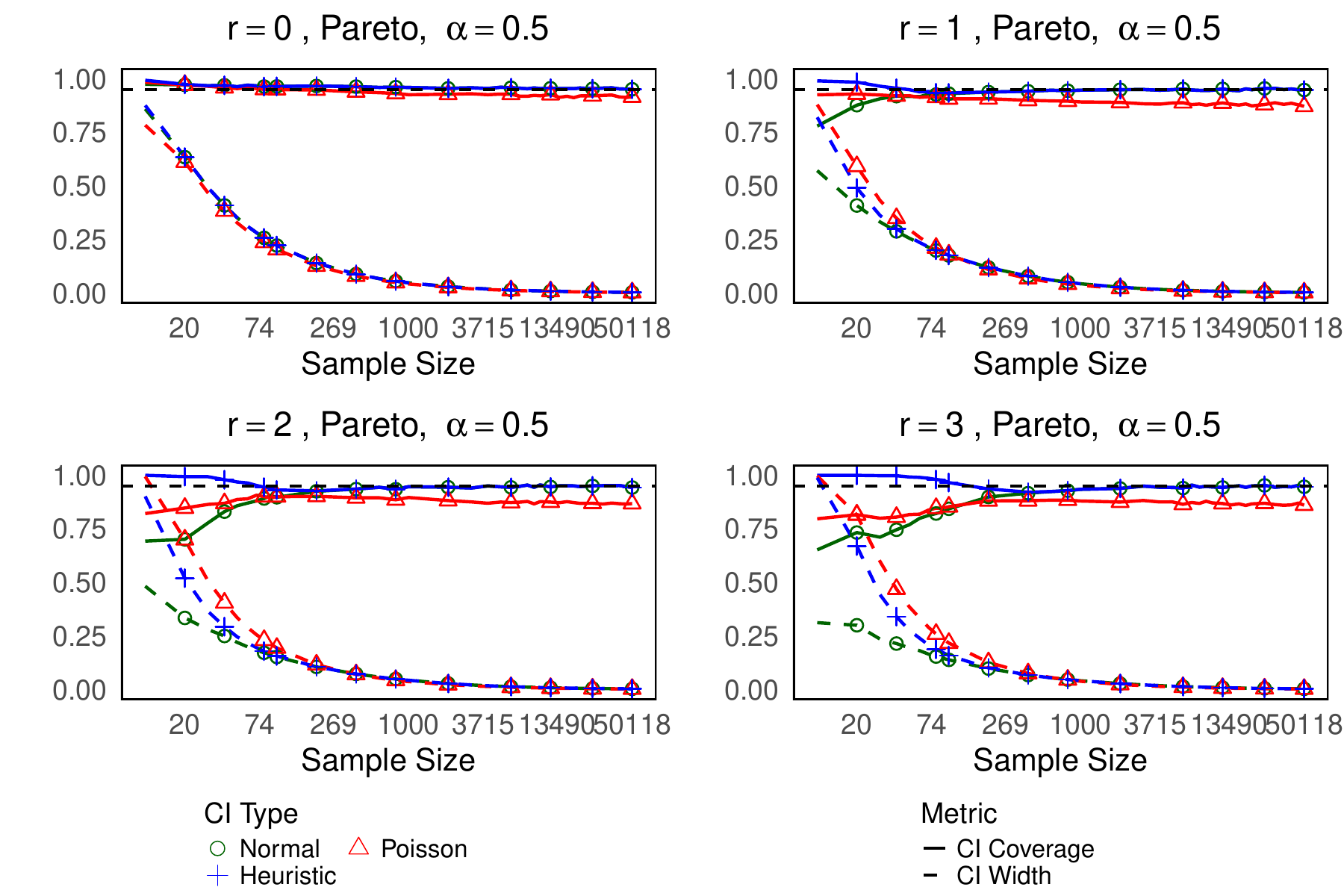}
	\caption[Results for the Discrete Pareto Distribution.]{
		\textbf{(a)} $\alpha = 0.5$. 
		Results for the Discrete Pareto Distribution. Each plot gives the coverage proportion and the mean width of three $95\%$ CIs for different choices of $r$ and $\alpha$. 
		These are based on $N=5000$ replications. The $x$-axis for sample size is displayed on a log (base 10) scale. The horizontal dashed line indicates the nominal level of $0.95$.}
	\label{fig:pareto}
\end{figure}

\begin{figure}[t]\ContinuedFloat
	\centering
	\includegraphics[width=\linewidth]{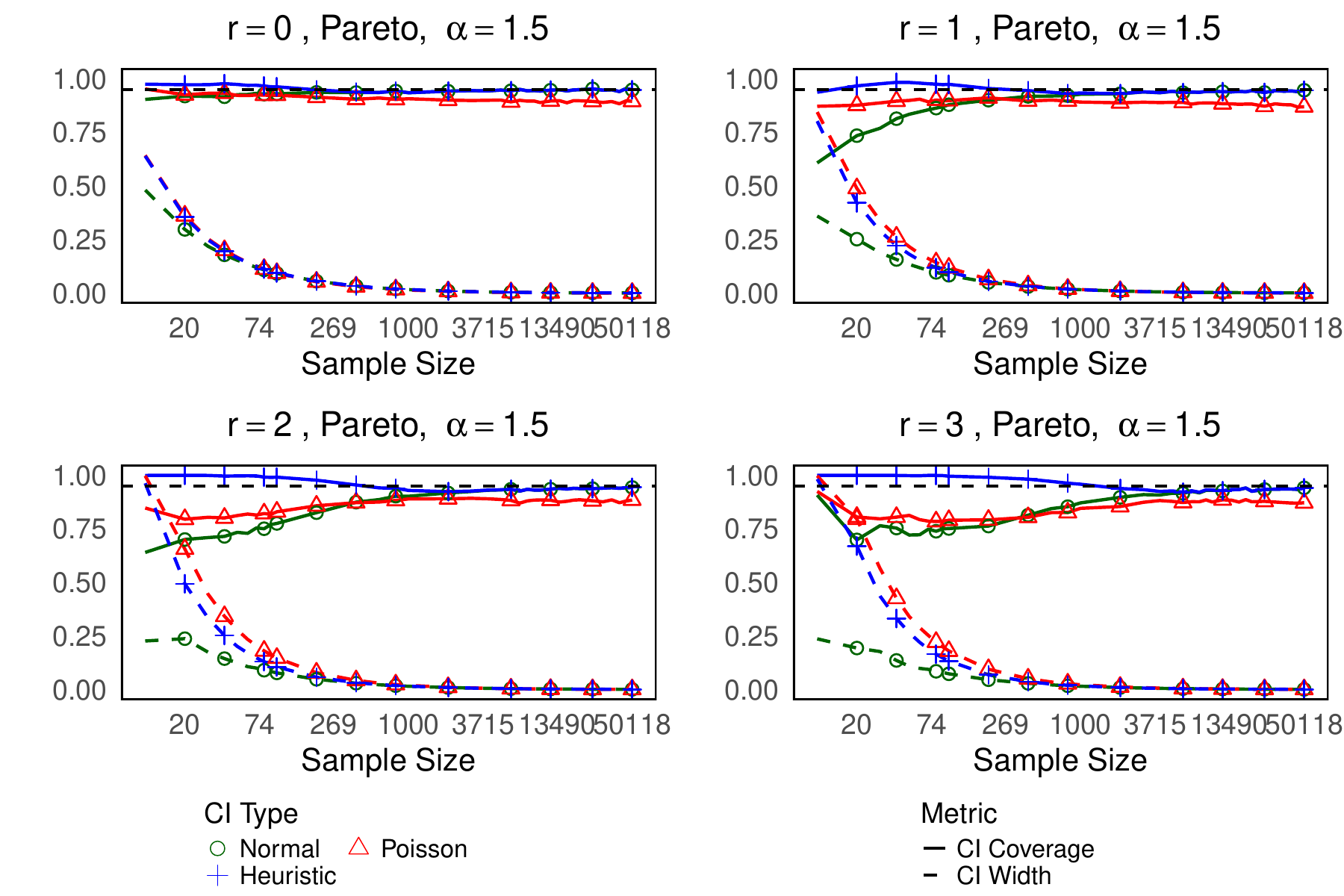}
	\caption[]{\textbf{(b)} $\alpha = 1.5$. Results for the Discrete Pareto Distribution, continued.}
\end{figure}

\begin{figure}[t]\ContinuedFloat
	\centering
	\includegraphics[width=\linewidth]{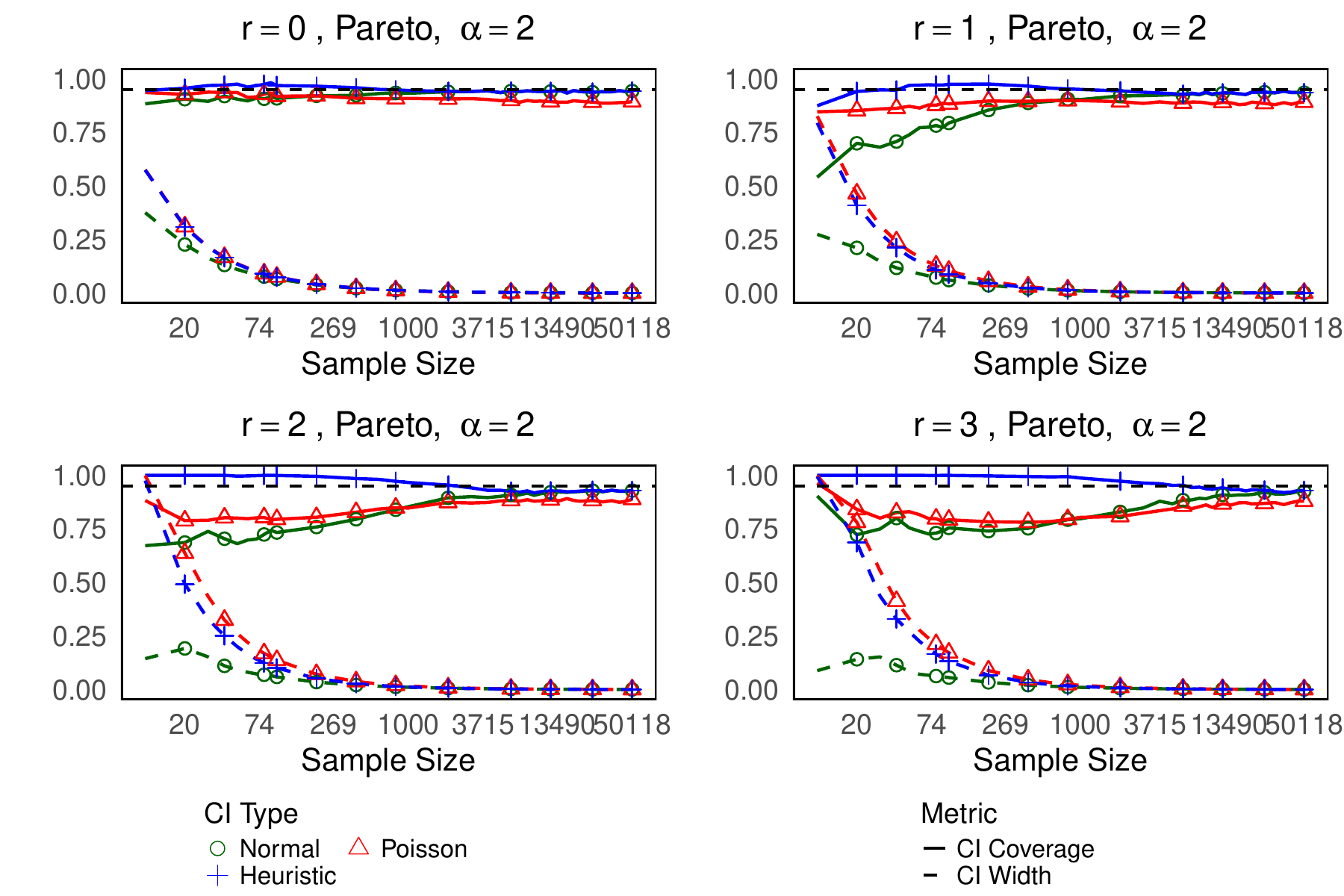}
	\caption[]{\textbf{(c)} $\alpha = 2$. Results for the Discrete Pareto Distribution, continued.}
\end{figure}

\section{CIs based on Concentration Inequalities}\label{sec: concentration ineq CIs}

In this paper, we primarily focus on CIs derived from asymptotic distributions. However, another strand of the literature considers confidence intervals constructed using concentration inequalities. In this section, we give a brief overview of these CIs and present a small simulation study comparing their performance with that of CIs based on asymptotic distributions.

Perhaps the earliest concentration inequality-based CI for $\pi_{r,n}$ is given in \cite{McAllester:Schapire:2000}. There it is shown that, for any $\alpha\in(0,1)$ and any $n>r$, we have
$$
\rP( |\pi_{r,n} - T^*_{r,n}|\le B^\mathrm{MS}_{r,n,\alpha} )\ge 1-\alpha
$$
where $T^*_{r,n}$ is the modified Turing estimator given in \eqref{eq: modified Turing} and 
$$
B^\mathrm{MS}_{r,n,\alpha}= \frac{r+2}{n-r} + \sqrt{\frac{2\log(3/\alpha)}{n}}\times\left(\frac{r+1}{1-r/n} +r + \sqrt{2r\log(3n/\alpha)} +2\log(3n/\alpha)\right).
$$
This leads to the $(1-\alpha)100\%$ CI for $\pi_{r,n}$  given by $T^*_{r,n}\pm B^\mathrm{MS}_{r,n,\alpha}$. Historically, this is an important result as it stimulated a lot of research and gave a rate of convergence. However, this interval is not useful for practical purposes since its radius tends to be very large. For instance, it is readily checked that for $\alpha=0.05$ and $r=0,1,2,$ or $3$, the radius does not go below $1$ until the sample size is well above $n=5000$. Thus, this CI is not informative unless the sample size is extremely large.

Another concentration inequality-based CI is given in \cite{Painsky:2022} and is valid only for $r\ge1$. Combining equation (15) and Theorem 5 in that paper, it follows that for any $\alpha\in(0,1)$ and $r\ge1$, we have
$$
\rP( |\pi_{r,n} - T_{r,n}|\le B^{\mathrm{Pa}}_{r,n,\alpha} )\ge 1-\alpha,
$$
where
$$
B^{\mathrm{Pa}}_{r,n,\alpha} = \sqrt{\frac{g(r)+o(1)}{n\alpha}}.
$$
Here $o(1)$ denotes a quantity (independent of $\alpha$) that approaches $0$ as $n\to\infty$ and
$$
g(r) = -\frac{1}{r+1}\left(\frac{r}{r+1}\right)^{2r} + \frac{e}{\sqrt{2\pi}} \left( \sqrt{r+1}\left(\frac{r}{r+1}\right)^{r} + \sqrt{r+2}\left(\frac{r+1}{r+2}\right)^{r+2} \right).
$$
This leads to the $(1-\alpha)100\%$ CI for $\pi_{r,n}$  given by
$$
T_{r,n}\pm \sqrt{\frac{g(r)}{n\alpha}}.
$$

The final concentration inequality-based CI is given in \cite{Favaro:Naulet:2024} and is valid only for $r=0$. Proposition 2 in that paper  states that for any $x>0$ and any $n\ge2$ we have
\begin{eqnarray*}
\rP\left( \max \{0,L_{n}(x)\} \le \pi_{0,n}  \right) \ge 1-6e^{-x} \mbox{ and } \rP\left( \pi_{0,n} \le \min\{1, U_{n}(x)\} \right) \ge 1-7e^{-x},
\end{eqnarray*}
where
\begin{eqnarray*}
L_{n}(x) &=& T_{0,n} - 2\left( \sqrt{\frac{1}{n(n-1)}} +\frac{\sqrt 2}{n}  \right) \sqrt{x C_{n}} - \left( 4 \sqrt{\frac{2}{n(n-1)}} + \frac{8+2/3}{n} +  \frac{32}{n(n-1)}\right)x\\
&&\qquad - \frac{4}{n(n-1)}  C_{n},\\
U_{n}(x) &=& T_{0,n} + 2\left( \sqrt{\frac{1}{n(n-1)}} +\frac{\sqrt 2}{n}  \right) \sqrt{x C_{n}} + \left( 4\sqrt{\frac{2}{n(n-1)}} + \frac{8+2/3}{n} \right)x \\
&&\qquad+ \frac{4\left(x+\log(n)\right)}{3n}x,
\end{eqnarray*}
and
$$
C_{n} = \sum_\ell 1_{[y_{\ell,n}\ge 1]}.
$$
Note that $C_{n}$ is the total number of distinct letters observed at least once in the sample. Using the Bonferroni inequality, we combine the above inequalities to get
$$
\rP\left( \max \{0,L_{n}(x)\} \le \pi_{0,n} \le \min\{1, U_{n}(x)\} \right) \ge 1-13e^{-x}.
$$
From here, we obtain a $(1-\alpha)100\%$ CI for $\pi_{0,n}$ as follows. First, we select $x$ such that $(1-\alpha) = 1-13e^{-x}$, or equivalently $x=\log(13/\alpha)$. Our CI is then $\left[\max \{0,L_{n}(x)\}, \min\{1, U_{n}(x)\}\right]$.

We now present the results of a small simulation study comparing the performance of the asymptotic CIs with those based on concentration inequalities. As with the asymptotic CIs, we truncate the concentration inequality-based CIs to ensure  that they stay within valid limits: if the calculated lower bound is less than $0$, we set it to $0$, and if the calculated upper bound exceeds $1$, we set it to $1$. We do not report results for the CI from \cite{McAllester:Schapire:2000}, as it is informative only for extremely large sample sizes. Similarly, in the interest of space, we only report results for the geometric and the discrete Pareto distributions, as the results for the other distributions were similar. For comparison, we use the Heuristic CI as the representative asymptotic CI.

The results are shown in Figures \ref{fig:geo_conin} and \ref{fig:pareto_conin}. We can see that the concentration-based CIs are substantially wider than the Heuristic CI, and in some cases they have a non-informative width of $1$ even for fairly large sample sizes. Moreover, their empirical coverage is consistently close to $100\%$, indicating that they are highly conservative. This is not surprising as these intervals are derived from concentration inequalities and, thus, their theoretical coverage is only guaranteed to be at least $(1-\alpha)100\%$, with no control on how much larger the actual coverage may be.

\begin{figure}[t]
	\centering
	\includegraphics[width=\linewidth]{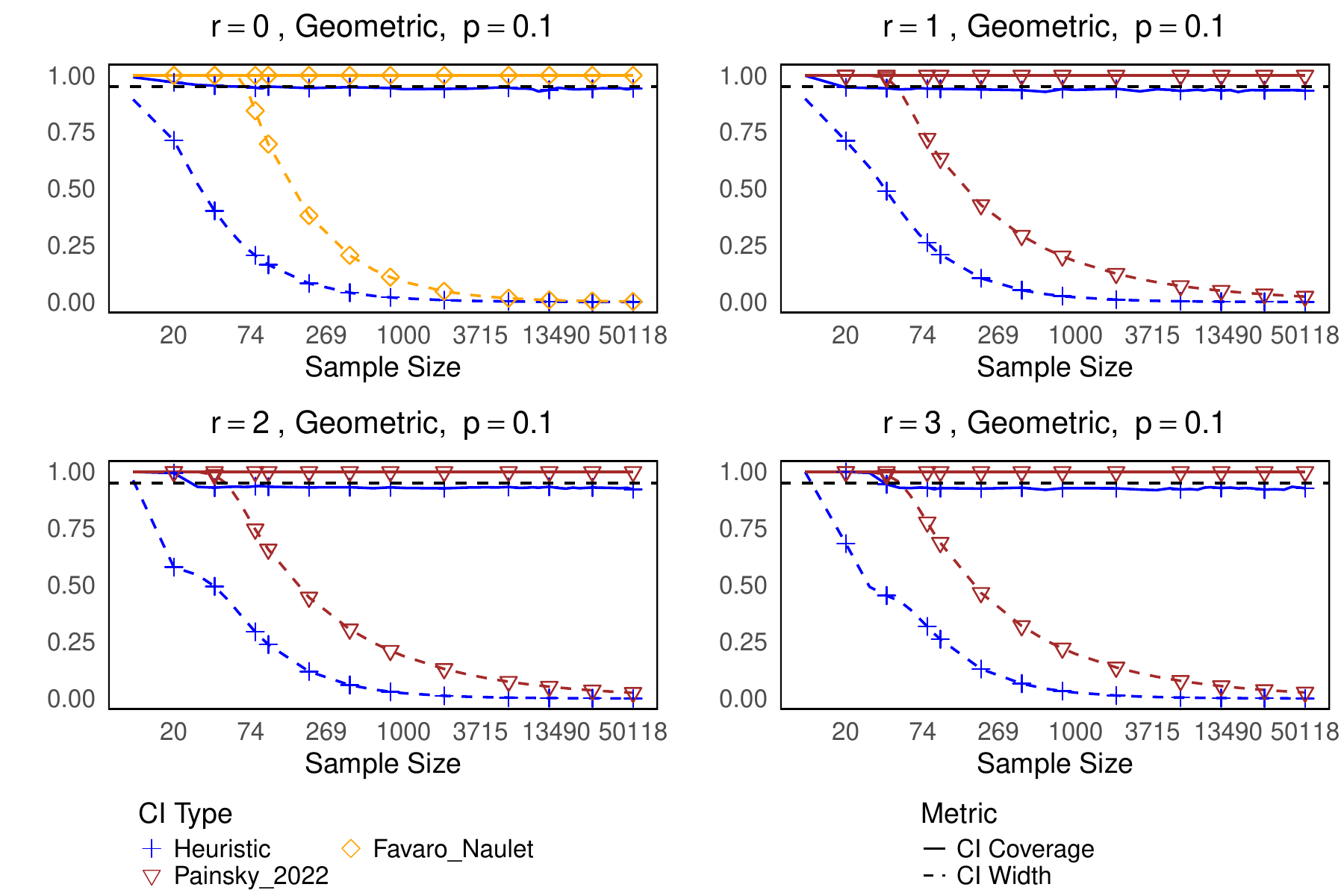}
	\caption[Results for the Geometric Distribution.]{
		\textbf{(a)} $p = 0.1$. 
		Results for the Geometric Distribution. Each plot gives the coverage proportion and the mean width of two $95\%$ CIs. These are based on $N=5000$ replications. The sample size on the $x$-axis is displayed on a log (base 10) scale. The horizontal dashed line indicates the nominal level of $0.95$.}
	\label{fig:geo_conin}
\end{figure}

\begin{figure}[t]\ContinuedFloat
	\centering
	\includegraphics[width=\linewidth]{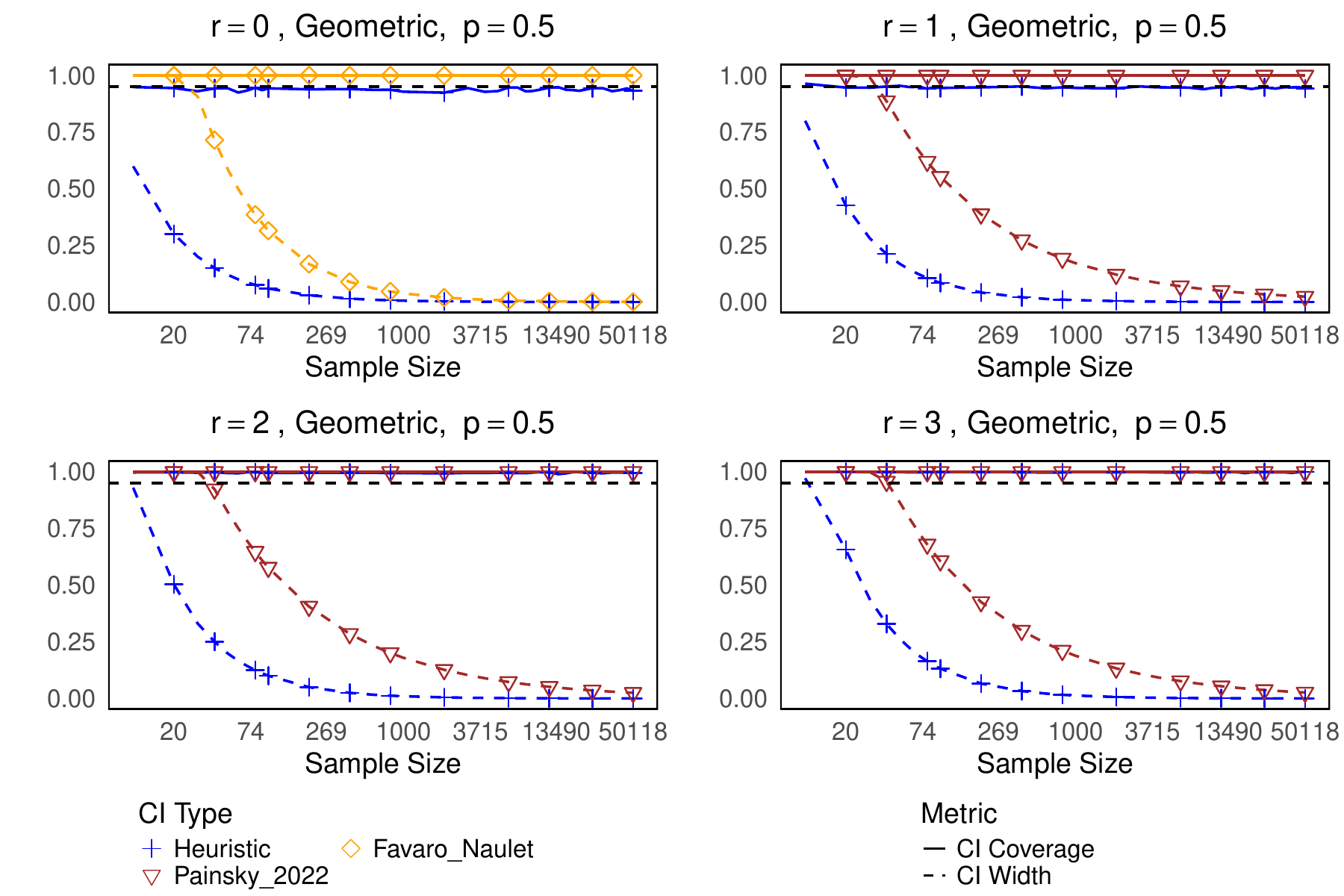}
	\caption[]{\textbf{(b)} $p = 0.5$. Results for the geometric distribution, continued.}
\end{figure}

\begin{figure}[t]\ContinuedFloat
	\centering
	\includegraphics[width=\linewidth]{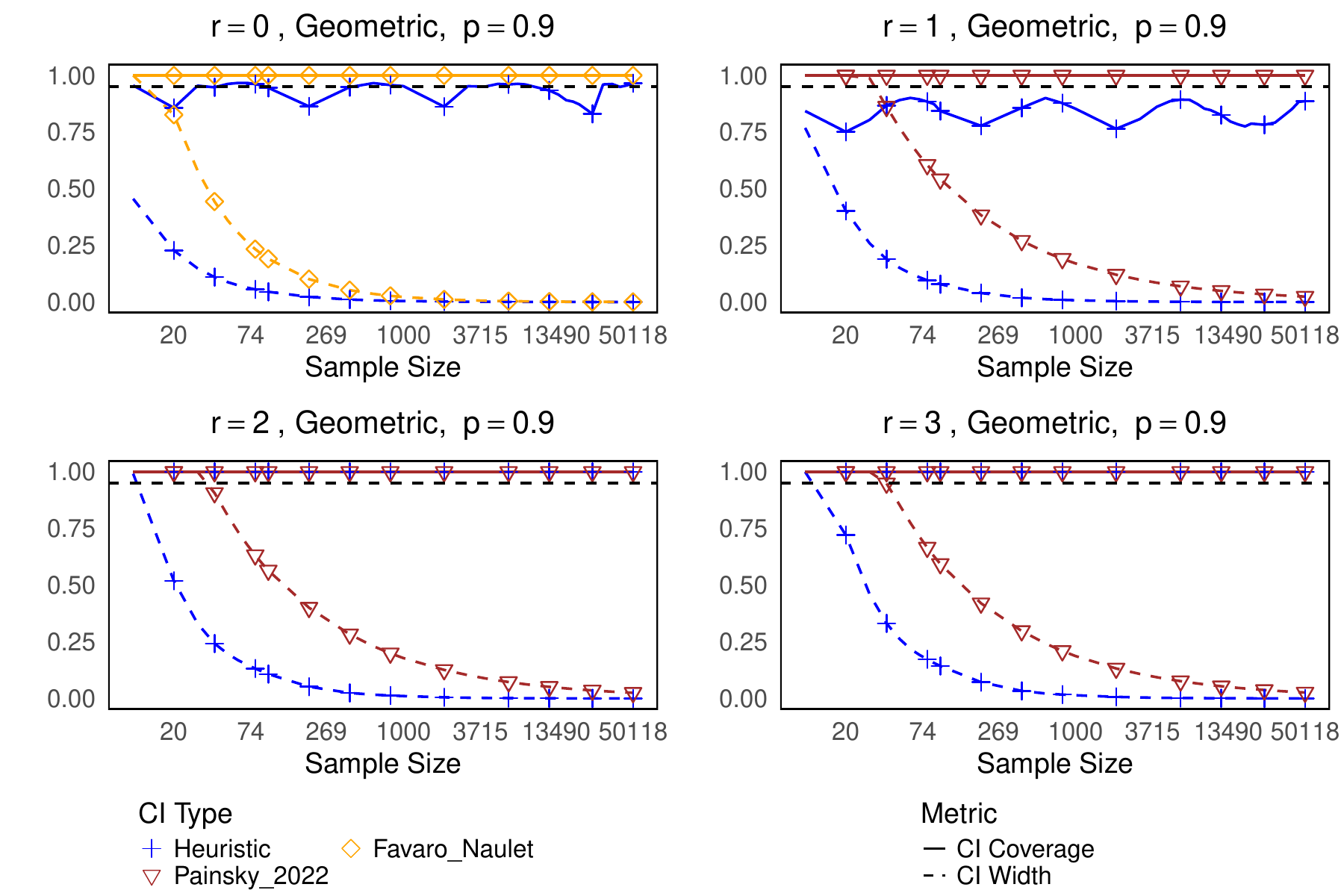}
	\caption[]{\textbf{(c)} $p = 0.9$. Results for the geometric distribution, continued.}
\end{figure}

\begin{figure}[t]\ContinuedFloat
	\centering
	\includegraphics[width=\linewidth]{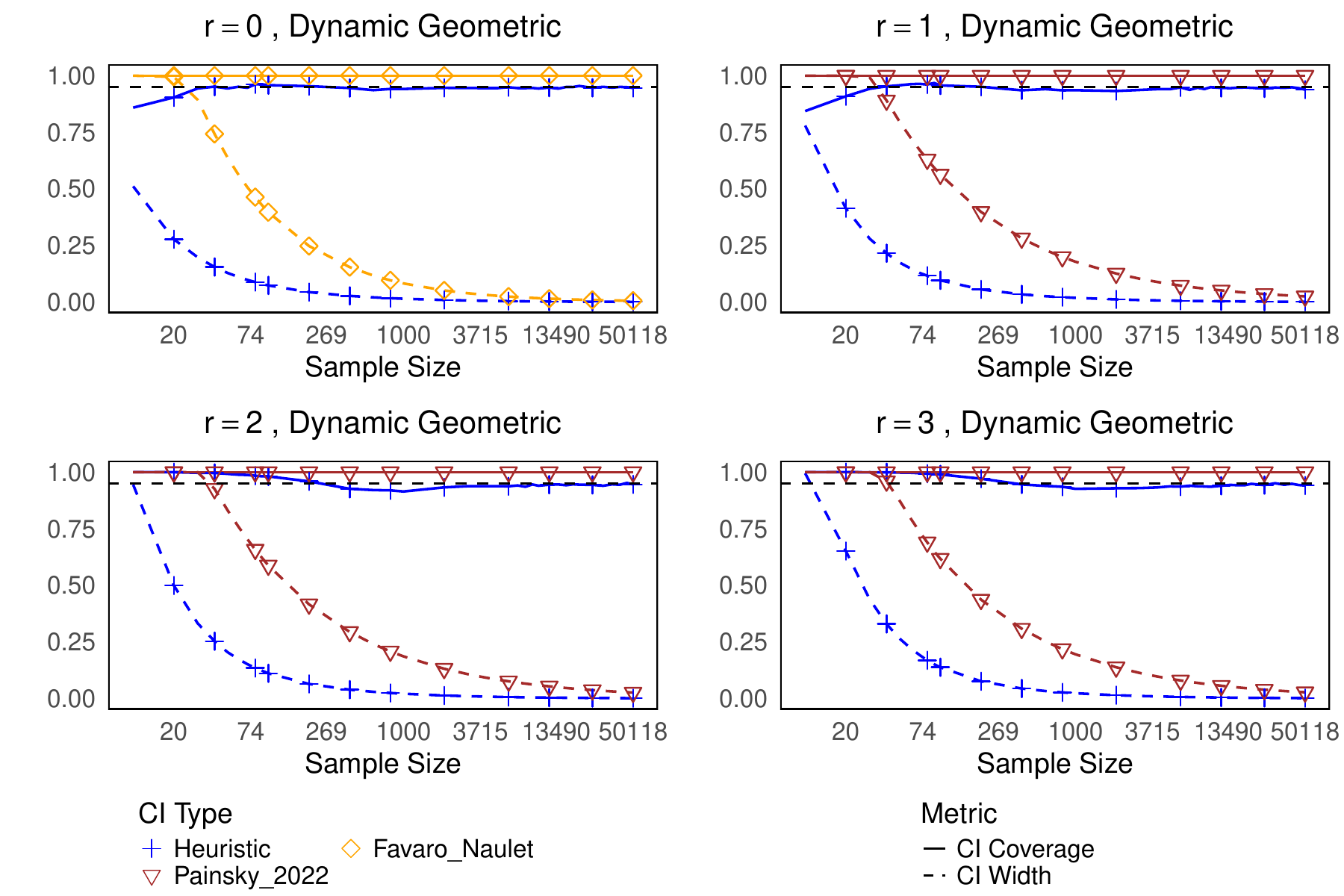}
	\caption[]{\textbf{(d)} Dynamic geometric distribution. Results for the geometric distribution, continued.}
\end{figure}

\begin{figure}[t]
	\centering
	\includegraphics[width=\linewidth]{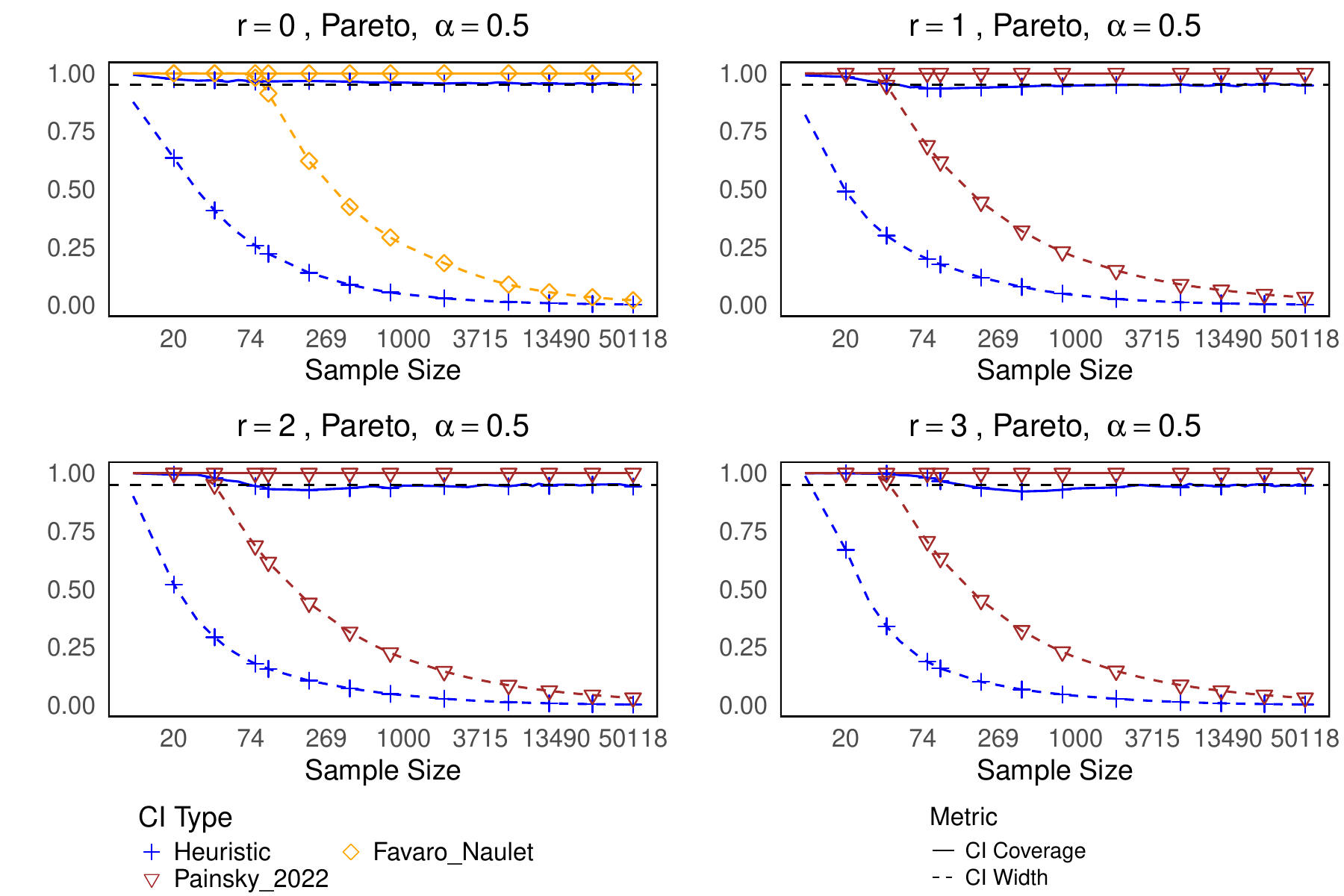}
	\caption[Results for the Discrete Pareto Distribution.]{
		\textbf{(a)} $\alpha = 0.5$. 
		Results for the Discrete Pareto Distribution. Each plot gives the coverage proportion and the mean width of two $95\%$ CIs for different choices of $r$ and $\alpha$. These are based on $N=5000$ replications. The $x$-axis for sample size is displayed on a log (base 10) scale. The horizontal dashed line indicates the nominal level of $0.95$.}
	\label{fig:pareto_conin}
\end{figure}

\begin{figure}[t]\ContinuedFloat
	\centering
	\includegraphics[width=\linewidth]{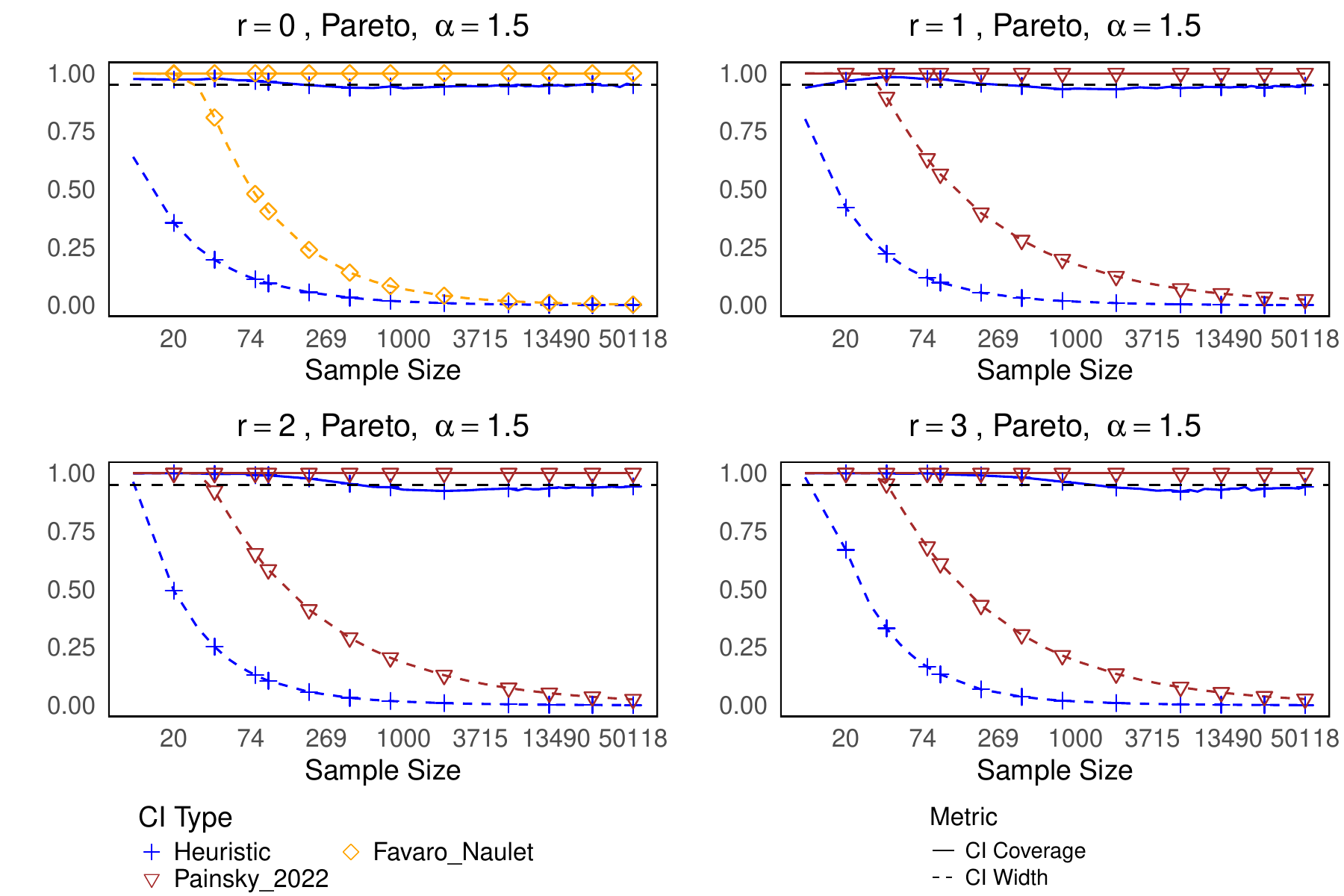}
	\caption[]{\textbf{(b)} $\alpha = 1.5$. Results for the discrete Pareto distribution, continued.}
\end{figure}

\begin{figure}[t]\ContinuedFloat
	\centering
	\includegraphics[width=\linewidth]{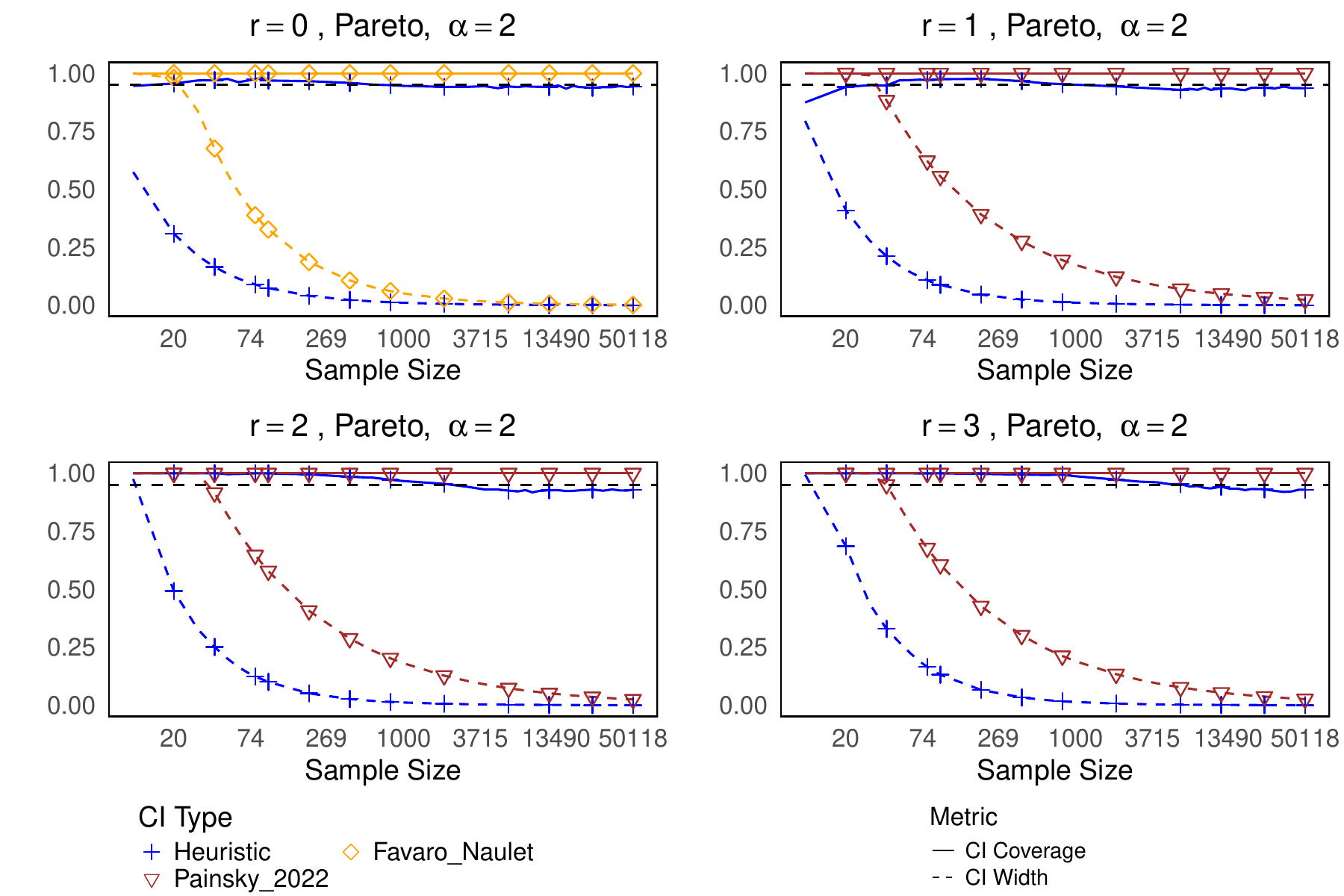}
	\caption[]{\textbf{(c)} $\alpha = 2$. Results for the discrete Pareto distribution, continued.}
\end{figure}

\section{Authorship Attribution and X (Twitter) Data}\label{sec: application}

In this section, we present an application to the problem of authorship attribution. For an overview, see, e.g., \cite{Zhang:Huang:2007}, \cite{Grabchak:Zhang:Zhang:2013}, \cite{Grabchak:Cao:Zhang:2018}, \cite{Zheng:Zheng:Kundu:2023}, and the references therein. In this context, $\mathcal A$ represents the common vocabulary consisting of all of the words that different authors can use. Associated with each author is a word-type distribution, $\mathcal P=\{p_\ell:\ell\in\mathcal A\}$, where for word $\ell\in\mathcal A$, $p_\ell$ is the probability with which the author will use $\ell$. In practice, it is usually difficult to estimate the word-type distribution. Instead, one typically estimates so-called diversity indices, which are functionals of the word-type distribution. Two of the most commonly used diversity indices are Shannon's entropy \citep{Grabchak:Zhang:Zhang:2013} and Simpson's index \citep{Grabchak:Cao:Zhang:2018}.

A common approach to authorship attribution is to consider two writing samples and to estimate one or more diversity indices for each sample. One then compares these estimated indices to see if they are statistically different from each other. Such approaches do not take into account the order in which words are written, only their relative frequencies. While some information is lost, this does not seem to be a major issue; see the discussion in \cite{Grabchak:Zhang:Zhang:2013}. We now introduce our methodology for using Turing's estimators for the problem of authorship attribution. It is motivated by, but different from, the approach of comparing diversity indices.

Let $n_1$ be the size of the first writing sample and let $n_2$ be the size of the second writing sample. We denote the first sample as the corpus and the second as the testing set. The idea is to check if the testing set has too many (or too few) words that appear rarely (or never) in the corpus. Toward this end, for the corpus, we construct $95\%$ CIs for $\pi_{r,n_1}$ for $r=0,1,2,\dots, R$, where $R<n_1$. Then, for each $r = 0, 1, 2,\dots, R$, we calculate the detecting points
\begin{align*}
	D_r = \frac{A_r}{n_2},
\end{align*}
where $A_r$ is the number of words in the testing set that are observed exactly $r$ times in the corpus. We allow for repetition. Thus, if a word appears several times in the testing set, then each instance is counted separately when calculating $A_r$. When $r=0$, then $A_0$ is just the number of words (including repetitions) in the testing set that do not appear in the corpus. Next, we compare the detecting points $D_r$ with the corresponding CI. If, for most values of $r$, $D_r$ falls inside the CI, it suggests that the writing samples are from the same author; while if most of them fall outside the CI, it suggests that the writing samples are from different authors.

To motivate this methodology, we note that $\pi_{r,n_1}$ is the probability of seeing a letter that is observed exactly $r$ times in the corpus. Thus, in a sample of size $n_2$, we expect to see $n_2\pi_{r,n_1}$ such words. It follows that
$$
n_2\pi_{r,n_1} = \rE[A_r| \mbox{the corpus}] 
$$
and hence
$$
\pi_{r,n_1} = \rE[D_r| \mbox{the corpus}] \approx D_r.
$$

\begin{figure}[ht]
	\begin{center}
		\includegraphics[width=0.8\linewidth,trim={0.6cm 4.0cm 0.65cm 4.0cm},clip]{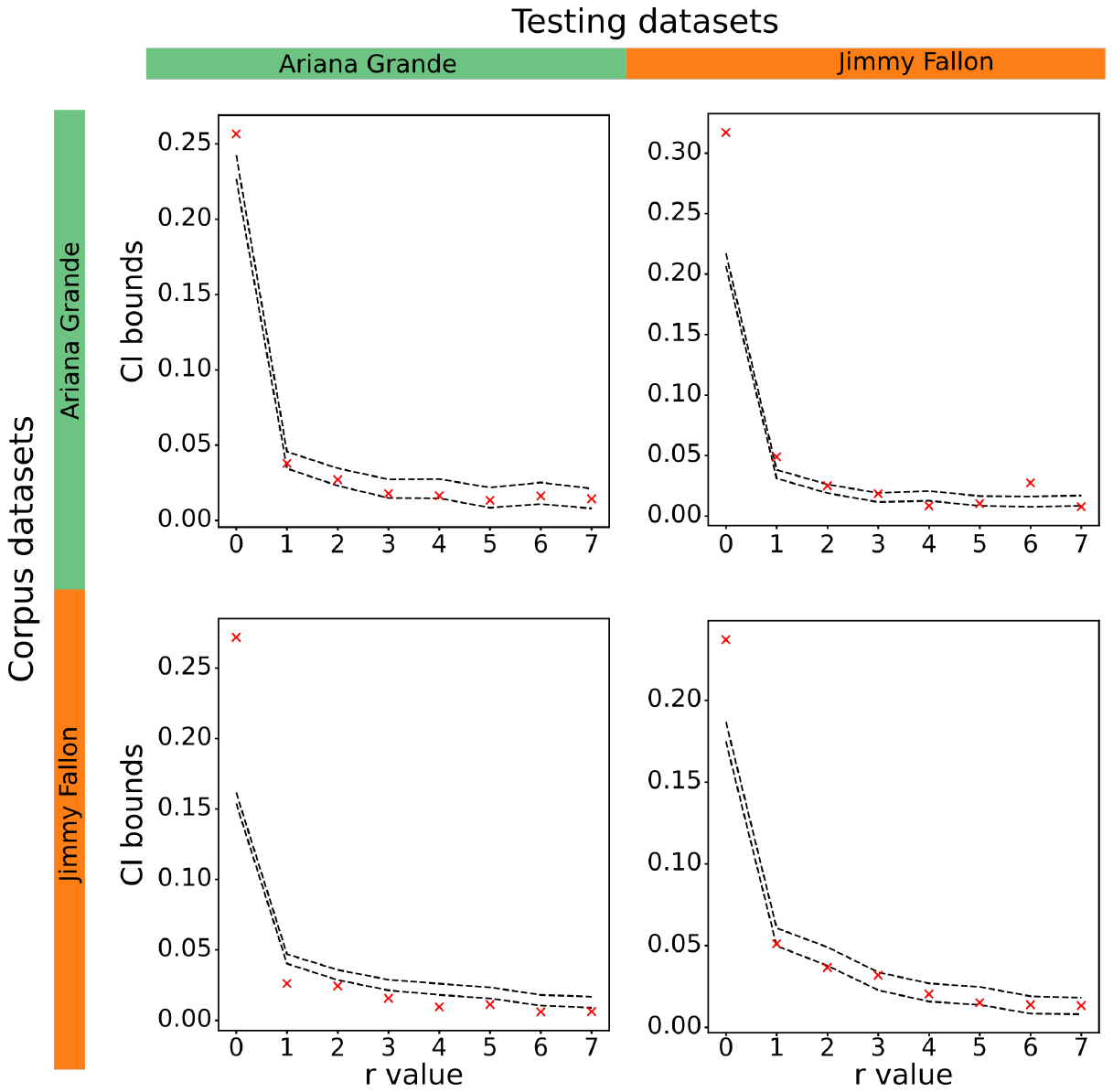}
		\caption{The CIs are constructed from the Corpus and the detecting points are calculated from the testing sets.}
		\label{fig: Application Grande}
	\end{center}
\end{figure}

\begin{figure}[ht]
	\begin{center}
		\includegraphics[width=\linewidth,trim={3.0cm 14.0cm 3.8cm 2.5cm},clip]{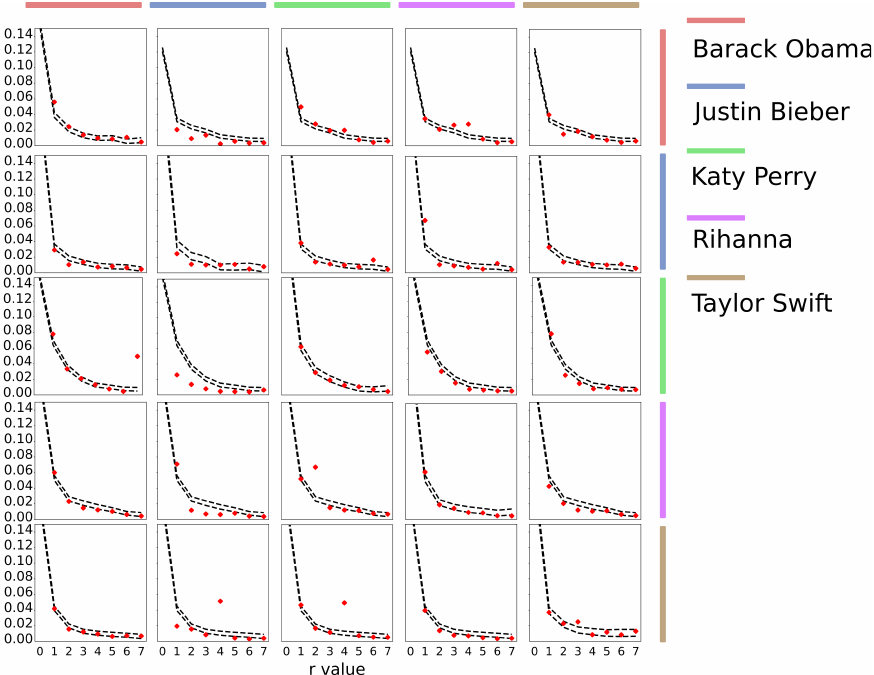}
		\caption{The CIs are constructed from the Corpus and the detecting points are calculated from the testing sets. The row indicates the corpus set and the column indicates the testing set.}
		\label{fig: Application 5}
	\end{center}
\end{figure}

We now illustrate our methodology by applying it to real-world data. For simplicity and comparability, we calculate all CIs using the Normal CI given in \eqref{eq: Normal CI}. The data come from X (formerly Twitter). The goal is to check if two X accounts are from the same author, which is important for the problem of identifying fake accounts. In this context, the writing samples are comprised of tweets from two accounts. These have been preprocessed to remove capitalization, punctuation, and URLs. Furthermore, retweets are excluded from the data. All data in this section were obtained from \cite{TwitterData}; see also \cite{TwitterData:paper}.

We begin by comparing two popular X users: Ariana Grande and Jimmy Fallon. The first sample consists of all tweets from Ariana Grande in 2015-2017 and the second consists of all tweets from Jimmy Fallon in 2013-2017. The sample from Ariana Grande contains $52647$ words and the one from Jimmy Fallon contains $36365$ words.

First, as a baseline, we compare each author's writing with writing of the same author. Toward this end, we randomly divide each writing sample into two parts. We use one part as the corpus to construct the CIs and the other as the testing set. The results are given in the plots on the diagonal in Figure \ref{fig: Application Grande}. We can see that most of the detecting points fall inside the CI, which indicates that the testing set is from the same author. The main exception is for $r=0$. It seems that we have too many words in the testing set that do not appear in the corpus. This may be part of the nature of X, where certain topics and people become part of the zeitgeist for a short period of time and are just referred to in one tweet. Another reason may be related to when asymptotic normality holds for $r=0$. Whatever the reason, these results suggest that, for authorship attribution in the context of X data, the use of values of $r \ge 1$ is more relevant and seems to work well.

Next, we use the full writing sample from each author as the corpus to construct the CIs and the full writing sample from the other author as the testing set. The results are given in the off-diagonal plots in Figure \ref{fig: Application Grande}. We see that most of the points fall outside the CI, which indicates that the testing set is from a different author. Note that things are not symmetric and that the detecting points seem to be further outside the CI when we use Jimmy Fallon's tweets as the corpus than when we use Ariana Grande's tweets.

To better illustrate the methodology, we compare five of the top X (then Twitter) users in 2017. The results are given in Figure \ref{fig: Application 5}. The plots do not include $r=0$, as these would make it hard to zoom in on the main parts of the plot, and, as we have seen, they do not seem to be relevant in the context of this application. Overall, the method seems to work well. However, in some cases, e.g., Justin Bieber, there are quite a few points that are outside the CI when comparing the authors to themselves. Similarly, there are situations where the method has difficulty distinguishing between two different authors, e.g., Justin Bieber and Taylor Swift.

\section{Discussion}\label{sec: disc}

In this paper, we performed a simulation study to compare the various CIs for the occupancy probabilities that appear in the literature. We saw that asymptotic CIs tend to work better than intervals based on concentration inequalities, as the latter tend to be overly conservative. We also introduced the Heuristic CI, which aims to select the most appropriate CI for a given random sample. Our simulation results indicate that it performs very well across a wide range of scenarios, including ones for which it is not known if the asymptotic results hold. This suggests that the Heuristic CI is robust in practice. Furthermore, we introduced a novel methodology for authorship attribution and applied it to several datasets from users of X (formerly Twitter). Along the way, we proved several theoretical results, which give more context to our simulations and are interesting in their own right. 

Throughout this paper, since the quantity of interest $\pi_{r,n}\in[0,1]$, we truncated all CIs to ensure that the lower bound is at least $0$ and that the upper bound is at most $1$. This does not affect the performance of the CI as $\pi_{r,n}$ is an element of the original CI if and only if it is an element of the truncated one. On the other hand, it makes it easier to interpret the width of the CI. For instance, if, after truncation, the width of the CI is $1$, then the interval covers all possible values of $\pi_{r,n}$ and is, thus, not informative.

We conclude the paper with a brief summary/discussion of our simulation results from the perspective of heavy-tails. It has been observed that Turing's estimator tends to work better for heavier-tailed distributions than for lighter-tailed ones. Empirically, this was seen in the simulation study conducted in \cite{Grabchak:Cosme:2017}, which considered the relative error of Turing's estimator with $r=0$. Theoretically, it is suggested by the fact that results for consistency and asymptotic normality of Turing's estimator are only known for relatively heavy tailed distributions, see, e.g., \cite{Ohannessian:Dahleh:2012}, \cite{Grabchak:Zhang:2017}, or \cite{Chang:Grabchak:2023}. The simulation results in the current paper echo this observation.

The discrete Pareto distribution has the heaviest tails among all of the distributions that we considered, and it is the only one of them for which asymptotic normality is known to hold without requiring a dynamic setting. The parameter $\alpha>0$ determines how heavy the tails are. The lower the value of $\alpha$, the heavier the tails. Our simulation results showed that the lower this value, the quicker the coverage proportion for the Normal CI approaches $0.95$. 

The geometric distribution has exponential tails, which are often considered to be on the boundary between heavy and light. The parameter $p\in(0,1)$ determines how heavy the tails are, with lower values of $p$ leading to heavier tails. Our simulation results showed that the lower this value, the closer the coverage proportions for both the Normal CI and the Poisson CI are to $0.95$. Asymptotic normality only holds in the dynamic case, where $p\to0$, i.e., as the tails get progressively heavier.

The discrete uniform distribution takes on only a finite number of possible values and is thus light-tailed. The parameter $K\ge1$ determines how heavy the tails are. The larger the value of $K$, the heavier the tails. Our simulation results showed that the larger this value, the longer the period during which the Normal CI is close to $0.95$ in its pre-limit apparent convergence phase. Furthermore, we observed that asymptotic normality and asymptotic Poissonity only hold in the dynamic case, where $K\to\infty$, which is, again, as the tails get progressively heavier. 

\backmatter

\bigskip

\begin{appendices}

\section{Conditions For Convergence}\label{sec: cond}

In this section, we summarize the known sufficient conditions for the convergence in \eqref{eq: asymp normal}, \eqref{eq: asymp ratio normal}, and \eqref{eq: asymp Pois} to hold. Let $\mathcal P_n=\{p_{\ell}^{(n)}:\ell\in\mathcal A\}$ be a sequence of probability distributions on alphabet $\mathcal A$. 

\begin{prop}\label{prop: suf cond}
	Let 
	$$
	s^2_{r,n} = \sum_{\ell\in\mathcal A} \left(r+1+np_{\ell}^{(n)}\right) e^{-np_{\ell}^{(n)}} \frac{\left(np_{\ell}^{(n)}\right)^{r+1}}{r!}.
	$$
	1. If $s_{r,n}\to\infty$, $s_{r,n}/\sqrt n\to0$, and 
	\begin{align}\label{eq: Lindeberg type condition}
		\lim_{n \rightarrow \infty}s_{r,n}^{-2}\sum_{\ell\in\mathcal A}e^{-n p_{\ell}^{(n)} }\left(n p_{\ell}^{(n)}\right)^{(r+2)}1_{\left[n p_{\ell}^{(n)} \geq \epsilon s_{r,n}\right]} = 0 \hspace{0.5cm} \forall \epsilon > 0,
	\end{align}
	then both \eqref{eq: asymp normal} and \eqref{eq: asymp ratio normal} hold.\\
	2. If $s_{r,n}\to c\in(0,\infty)$ and \eqref{eq: Lindeberg type condition} holds, then \eqref{eq: asymp Pois} holds with $c^*=c^2/(r+1)^2$.\\
	3. If $r=0$, $s_{0,n}\to\infty$, $\limsup_{n\to\infty} \rE[N_{1,n}]/n <1$, and \eqref{eq: Lindeberg type condition} holds, then \eqref{eq: asymp normal} holds.
\end{prop}

The first two parts are given in Theorems 2.1 and 2.3 of \cite{Chang:Grabchak:2023}. The third is given in Theorem 1 of \cite{Zhang:Zhang:2009}, see also \cite{Esty:1983}. The following sufficient condition for \eqref{eq: Lindeberg type condition} can be found in, e.g., Proposition 2.2 of \cite{Chang:Grabchak:2023}.

\begin{lemma}
	If  $s_{r,n}/\log n\to\infty$, then \eqref{eq: Lindeberg type condition} holds.
\end{lemma}

\section{Proofs}\label{sec: proofs}

\subsection{Asymptotics for the Dynamic Uniform}\label{sec: dynamic uniform theory}

In this section, we derive asymptotic results for the dynamic uniform distribution. Let $\mathcal P_n=\{p_{\ell}^{(n)}:\ell\in\mathcal A\}$ be as in \eqref{eq: dynamic unif} and note that, in this case,
\begin{eqnarray*}
	s^2_{r,n} &=& \lfloor n^\gamma\rfloor \left(r+1+\frac{n}{\lfloor n^\gamma\rfloor} \right)e^{-n/\lfloor n^{\gamma}\rfloor}\frac{\left(n/  \lfloor n^\gamma\rfloor\right)^{r+1}}{r!}.
\end{eqnarray*}
It is readily checked that for any $\gamma>0$
$$
0\le \frac{n}{\lfloor n^{\gamma}\rfloor} - n^{1-\gamma} = \frac{n^{1-\gamma}}{\lfloor n^{\gamma}\rfloor}\left(n^\gamma-\lfloor n^{\gamma}\rfloor\right) \le \frac{n^{1-\gamma}}{\lfloor n^{\gamma}\rfloor} \sim n^{1-2\gamma}.
$$
Thus, for $\gamma>1/2$, we have $e^{-n/\lfloor n^{\gamma}\rfloor}\sim e^{-n^{1-\gamma}}$. It follows that for $\gamma=1$ we have 
$$
s^2_{r,n} \sim n \frac{(r+2)}{r!} e^{-1}
$$ 
and for $\gamma>1$ we have
$$
s^2_{r,n} \sim n^{r+1-\gamma r}\frac{\left(r+1\right) }{r!}.
$$

When $\gamma\in(0,1)$, we have 
\begin{eqnarray*}
	s^2_{r,n} &\le&  (n^\gamma -1)^{-r}  \left(r+1+\frac{n}{ n^\gamma-1} \right)e^{-n/( n^{\gamma}+1)}\frac{n^{r+1}}{r!}\to0
\end{eqnarray*}
and the conditions in Proposition \ref{prop: suf cond} do not hold. 

Next, consider the case where $\gamma=1$. We have $s_{r,n}\to\infty$ and $s_{r,n}/\log n\to\infty$, but $s_{r,n}/\sqrt n\to \sqrt{\frac{(r+2)}{r!} }e^{-1/2}$. Thus Proposition \ref{prop: suf cond} can only give us results for $r=0$. In this case,
\begin{eqnarray*}
	\frac{\rE[N_{1,n}]}{n} = \frac{\sum_{\ell\in\mathcal A} {n\choose 1} p_{\ell}^{(n)} \left(1-p_{\ell}^{(n)}\right)^{n-1}}{n} =   (1-n^{-1})^{n-1}\to e^{-1}<1.
\end{eqnarray*}
From here, Proposition \ref{prop: suf cond} implies that \eqref{eq: asymp normal} holds.  

Finally, we turn to the case $\gamma>1$. When $r=0$, we have $s_{r,n}/\sqrt n\to \sqrt{\frac{(r+2)}{r!} }$ and 
\begin{eqnarray*}
	\frac{\rE[N_{1,n}]}{n} =  (1-1/\lfloor n^{\gamma}\rfloor)^{n-1}\to 1.
\end{eqnarray*}
Thus, the conditions in Proposition \ref{prop: suf cond} do not hold. Henceforth assume that $r\ge1$. When $\gamma>1/r+1$ we have $s^2_{r,n}\to0$ and the conditions in Proposition \ref{prop: suf cond} do not hold. When $\gamma=1/r+1$, we have $s^2_{r,n}\to \frac{\left(r+1\right) }{r!}$. It follows that for any $\epsilon>0$ and large enough $n$
$$
1_{\left[n^{1-\gamma}>\epsilon s_{r,n}\right]} = 0.
$$
Hence, \eqref{eq: Lindeberg type condition} holds and Proposition \ref{prop: suf cond} implies that \eqref{eq: asymp Pois} holds with $c^*=1/(r+1)!$. When $\gamma<1/r+1$, we have $s^2_{r,n}\to\infty$, $s_{r,n}/\log n\to\infty$, and $s_{r,n}/\sqrt n\to0$. Thus, Proposition \ref{prop: suf cond} implies that \eqref{eq: asymp normal} and \eqref{eq: asymp ratio normal} hold in this case.

\subsection{Proof of Proposition \ref{prop: norm for dyn geo}}\label{eq: proof prop geo}

Let $f_n(x)=(e^{1/a_n}-1)e^{-x/a_n}$ for $x\in\mathbb R$ and note that $p_{\ell}^{(n)}=f_n(\ell)$ for $\ell=1,2,\dots$. Next, for $\eta>0$, let $g_\eta(x) = e^{-x}x^{\eta}$ for $x\ge0$ and let $h_{n,\eta}(x) = g_\eta(nf_n(x))$ for $x\in\mathbb R$. Note that $s^2_{r,n} =\frac{1}{r!} \sum_{\ell=1}^\infty \left((r+1)h_{n,r+1}(\ell)+ h_{n,r+2}(\ell)\right)$. Since $g'_\eta(x) = x^{\eta-1}e^{-x}(\eta-x)$, $g_\eta$ is increasing on $(0, \eta)$, decreasing on $(\eta, \infty)$, and $\max_{x\ge0}g_\eta(x) = g_\eta(\eta)$. Next, let $x_{n}^{*\eta}=-a_n \log\left(\eta n^{-1}(e^{1/a_n}-1)^{-1}\right)$ and note that $nf_n(x_n^{*\eta})=\eta$. Since $f_n$ is monotonically decreasing, it follows that $h_{n,\eta}$ is increasing on $(-\infty,x_n^{*\eta})$ and decreasing on $(x_n^{*\eta},\infty)$. 

\begin{lemma}\label{lemma: asym s}
	Assume that $a_n\to\infty$ with $a_n/n\to0$. \\
	1. We have
	\begin{eqnarray}\label{eq: to 0}
		\lim_{n\to\infty} n\left(e^{1/a_n}-1\right) = \lim_{n\to\infty} n\left(1-e^{-1/a_n}\right) = \infty.
	\end{eqnarray}
	2. For any $\eta>0$, we have
	$$
	\sum_{\ell=1}^\infty h_{n,\eta}(\ell) \sim a_n \Gamma(\eta).
	$$
	3. For any integer $r\ge0$, we have 
	$
	s^2_{r,n} \sim 2a_n(r+1).
	$
\end{lemma}

\begin{proof}
	First, note that 4.2.33 in \cite{Abramowitz:Stegun:1972} implies that for large enough $n$, 
	we have
	$$
	n\left(e^{1/a_n}-1\right) \ge  \frac{n}{a_n}\to\infty
	$$
	and hence
	$$
	n\left(1-e^{-1/a_n}\right) =  e^{-1/a_n} n\left(e^{1/a_n}-1\right) \to\infty.
	$$
	Next, note that $x_n^{*\eta}\to\infty$, $h_{n,\eta}(x_n^{*\eta})/a_n = g_\eta(\eta)/a_n\to0$, and $h_{n,\eta}(1)/a_n $ $=$ $g_{\eta}(n(1-e^{-1/a_n}))/a_n $ $\to$ $0$, where we use \eqref{eq: to 0} and the fact that $g_\eta(x)\to 0$ as $x\to\infty$. By the Euler-Maclaurin Lemma, see e.g.\ Lemma 1.6 in \cite{Zhang:2017}, we have
	\begin{eqnarray*}
		\lim_{n\to\infty} \frac{\sum_{\ell=1}^\infty h_{n,\eta}(\ell)}{a_n \Gamma(\eta)} &=& \lim_{n\to\infty} \frac{\int_1^\infty h_{n,\eta}(x) \rd x}{a_n \Gamma(\eta)}\\
		&=&  \lim_{n\to\infty}  \frac{a_n\int_0^{n\left(e^{1/a_n}-1\right)} e^{-t} t^{\eta-1} \rd t}{a_n \Gamma(\eta)} = 1,
	\end{eqnarray*}
	where we use the change of variables $t=n f_n(x)$ and \eqref{eq: to 0}. Turning to the third part, we have
	\begin{eqnarray*}
		\frac{s_{r,n}^2}{a_n(r+1)} &=& \frac{\left((r+1)\sum_{k=1}^\infty h_{n,r+1}(k) +\sum_{k=1}^\infty h_{n,r+2}(k) \right)/r!}{2a_n(r+1)} \\
		&\to& \frac{\left((r+1)\Gamma(r+1) +\Gamma(r+2) \right)/r!}{2(r+1)}=1,
	\end{eqnarray*}
	which completes the proof.
\end{proof}

\begin{lemma}\label{lemma: asym s finite}
	Assume that $a_n\to a\in[0,\infty)$. \\
	1. For any $\eta>0$, we have
	$$
	a\Gamma(\eta)-2 e^{-\eta}\eta^\eta \le \liminf_{n\to\infty} \sum_{\ell=1}^\infty h_{n,\eta}(\ell) \le\limsup_{n\to\infty} \sum_{\ell=1}^\infty h_{n,\eta}(\ell) \le a \Gamma(\eta) + 2 e^{-\eta}\eta^\eta.
	$$
	2. For any integer $r\ge0$, \eqref{eq: bounds for fixed geo} holds. 
\end{lemma}

\begin{proof}
	By arguments as in the proof of Lemma 1.6 in \cite{Zhang:2017}, we have
	\begin{eqnarray*}
		\sum_{k=1}^\infty h_{n,\eta}(k) &\le& \int_1^\infty h_{n,\eta}(x)\rd x + 2 h_{n,\eta}(x_n^{*\eta}) \\
		&=&  a_n\int_0^{n\left(e^{1/a_n}-1\right)} e^{-t} t^{\eta-1} \rd t + 2 g_{\eta}(\eta) \to a\Gamma(\eta)+2 e^{-\eta}\eta^\eta,
	\end{eqnarray*}
	where the second line follows by a change of variables as in the proof of Lemma \ref{lemma: asym s}. Similarly,
	\begin{eqnarray*}
		\sum_{k=1}^\infty h_{n,\eta}(k) &\ge& \int_1^\infty h_{n,\eta}(x)\rd x - 2 h_{n,\eta}(x_n^{*\eta}) 
		\to a\Gamma(\eta)-2 e^{-\eta}\eta^\eta.
	\end{eqnarray*}
	From here the second part follows immediately. 
\end{proof}

Lemma \ref{lemma: asym s finite} gives the second part of Proposition \ref{prop: norm for dyn geo}. We now turn to the first part. Henceforth, assume that $a_n\to\infty$ and $a_n/n\to0$.  Lemma \ref{lemma: asym s} implies that $s_{r,n}\to\infty$ and $s_{r,n}/\sqrt n\to0$. By Proposition \ref{prop: suf cond}, it suffices to show that \eqref{eq: Lindeberg type condition} holds. Toward this end, note that $nf_n(x)\le n(e^{1/a_n}-1)e^{-1/a_n}$ for $x\ge1$, thus when $\epsilon s_{r,n}> n(e^{1/a_n}-1)e^{-1/a_n}$, we have $1_{[n p_{\ell}^{(n)} \geq \epsilon s_{r,n}]} = 0$ for each $\ell=1,2,\dots$. Henceforth, assume that $\epsilon s_{r,n}\le n(e^{1/a_n}-1)e^{-1/a_n}$. In this case, we have $n f_n(x)\ge \epsilon s_{r,n}$ if and only if $x\le x_n^{**}$, where $x_n^{**}=-a_n\log(\epsilon s_{r,n} n^{-1} (e^{1/a_n}-1)^{-1})$. Furthermore, since $s_{r,n}\to\infty$, we have $x^{**}_{n}\le x_n^{*(r+2)}$ for large enough $n$. It follows that, for such $n$, $h_{n,r+2}$ is monotonically increasing for $x\le x^{**}$ and
\begin{eqnarray*}
	&&s_{r,n}^{-2}\sum_{\ell=1}^{\infty}e^{-n p_{\ell}^{(n)}}\left(n p_{\ell}^{(n)}\right)^{r+2}1_{[n p_{\ell}^{(n)} \geq \epsilon s_{r,n}]} \\
	&&\qquad = s_{r,n}^{-2}\sum_{\ell=1}^{\lfloor x_n^{**}\rfloor-1}  h_{n,r+2}(\ell) 
	+ s_{r,n}^{-2} h_{n,r+2}(\lfloor x_n^{**}\rfloor)\\
	&&\qquad \le s_{r,n}^{-2}\int_1^{x_n^{**}} h_{n,r+2}(x) 
	\rd x + s_{r,n}^{-2} g_{r+2}(\epsilon s_{r,n}) \\
	&&\qquad = s_{r,n}^{-2}a_n \int_{\epsilon s_{r,n}}^{n(1-e^{-1/a_n})} e^{-t}t^{r+1} \rd t + s_{r,n}^{-2} g_{r+2}(\epsilon s_{r,n}) \\\
	&&\qquad\le s_{r,n}^{-2}a_n \int_{\epsilon s_{r,n}}^{\infty} e^{-t}t^{r+1} \rd t + s_{r,n}^{-2} g_{r+2}(\epsilon s_{r,n}) \to 0,
\end{eqnarray*}
where we use  the change of variables $t=n f_n(x)$, dominated convergence, the fact that $s_{r,n}\to\infty$, the fact that $g_{r+2}(x)\to0$ as $x\to\infty$, and Lemma \ref{lemma: asym s}.

\end{appendices}

\section*{Acknowledgments}

The authors wish to thank the anonymous reviewers and the anonymous Associated Editor, whose comments led to an improvement in the presentation of this paper.

\section*{Conflict of Interest (COI)}

A previous version of this manuscript appeared as a preprint (see \cite{chang2025confidence}). Michael Grabchak is a member of the {\em Journal of Applied Statistics} editorial board.

\section*{Funding Statement}

No funding was received.

\section*{Data Availability}

The simulation code is available from the corresponding author upon reasonable request.


\end{document}